\documentclass[11pt,a4paper,twoside,reqno]{amsart}

\usepackage[
  a4paper,
  margin=3.2cm,
  includeheadfoot,
  headheight=14pt,
  footskip=18pt
]{geometry}
\usepackage[T1]{fontenc}
\usepackage[utf8]{inputenc}
\usepackage{amsmath,amssymb,amsfonts,amsthm}
\usepackage{mathtools}
\usepackage[mathscr]{eucal}
\usepackage{microtype}
\usepackage{enumitem}
\usepackage{aliascnt}
\usepackage{xcolor}
\usepackage[
  backend=biber,
  style=numeric,
  citestyle=numeric-comp,
  sorting=nyt,
  giveninits=true
]{biblatex}
\usepackage{hyperref}
\usepackage[nameinlink,capitalise,noabbrev]{cleveref}

\hypersetup{
  colorlinks=true,
  linkcolor=blue,
  citecolor=blue,
  urlcolor=blue,
  pdftitle={Stability of Diophantine Minimal Foliations in Arbitrary Codimension},
  pdfauthor={Hoan Nguyen}
}

\numberwithin{equation}{section}

\theoremstyle{plain}
\newtheorem{Th}{Theorem}[section]
\newaliascnt{Lemma}{Th}
\newtheorem{Lemma}[Lemma]{Lemma}
\aliascntresetthe{Lemma}
\newaliascnt{Cor}{Th}
\newtheorem{Cor}[Cor]{Corollary}
\aliascntresetthe{Cor}
\newaliascnt{Prop}{Th}
\newtheorem{Prop}[Prop]{Proposition}
\aliascntresetthe{Prop}

\theoremstyle{definition}
\newaliascnt{Def}{Th}
\newtheorem{Def}[Def]{Definition}
\aliascntresetthe{Def}
\newaliascnt{Rem}{Th}
\newtheorem{Rem}[Rem]{Remark}
\aliascntresetthe{Rem}

\crefname{Th}{Theorem}{Theorems}
\Crefname{Th}{Theorem}{Theorems}
\crefname{Lemma}{Lemma}{Lemmas}
\Crefname{Lemma}{Lemma}{Lemmas}
\crefname{Cor}{Corollary}{Corollaries}
\Crefname{Cor}{Corollary}{Corollaries}
\crefname{Prop}{Proposition}{Propositions}
\Crefname{Prop}{Proposition}{Propositions}
\crefname{Def}{Definition}{Definitions}
\Crefname{Def}{Definition}{Definitions}
\crefname{Rem}{Remark}{Remarks}
\Crefname{Rem}{Remark}{Remarks}

\newcommand{\T}{\mathbb T}
\newcommand{\R}{\mathbb R}
\newcommand{\Z}{\mathbb Z}

\newcommand{\cH}{\mathcal H}
\newcommand{\cP}{\mathcal P}
\newcommand{\cM}{\mathcal M}
\newcommand{\cG}{\mathcal G}
\newcommand{\Id}{\mathrm{Id}}
\newcommand{\df}{\mathrm{df}}
\newcommand{\norm}[2][]{\left\lVert #2\right\rVert_{#1}}
\newcommand{\ip}[2]{\left\langle #1,#2\right\rangle}
\newcommand{\Sym}{\operatorname{Sym}}
\newcommand{\so}{\mathfrak{so}}
\newcommand{\Hom}{\operatorname{Hom}}

\title[Stability of Diophantine minimal foliations]
{Stability of Diophantine Minimal Foliations\\
in Arbitrary Codimension}
\author{Hoan Nguyen}
\address{Department of Mathematics, The University of Chicago, Chicago, IL 60615}
\email{nth@uchicago.edu}
\date{}

\begin{document}

\begin{abstract}
We generalize Moser's stability theorem for minimal hypersurface foliations
to arbitrary codimension. Let $\T^N=\R^N/\Z^N$ be the standard torus, and let
$P\subset\R^N$ be an $m$-plane satisfying a Diophantine condition relative
to $\Z^N$. We prove that, under every sufficiently small smooth
$\Z^N$-periodic perturbation of the Euclidean metric, the affine foliation
of $\T^N$ induced by $P$ persists as a smoothly conjugate
foliation by minimal $m$-dimensional submanifolds. The normalized conjugating
displacement depends smoothly and tamely on the metric and is locally unique.
Moreover, the perturbed foliation admits a common smooth calibration, hence the lifted leaves are area-minimizing. The proof uses
Hamilton's formulation of the Nash--Moser implicit function theorem.
\end{abstract}

\maketitle

\section{Introduction}
\label{sec:introduction}

Let $\T^N=\R^N/\Z^N$ be the standard torus and let
$g_0=\sum_{\alpha=1}^N dz_\alpha^2$ be its Euclidean metric.
We study foliations induced by affine
$m$-planes, where $1\le m<N$. For any $m$-plane $P\subset \mathbb{R}^N$, the parallel translations $P+y, y\in P^\perp$ foliate $\mathbb{R}^N$. Following Moser
\cite{Moser1988}, we call such a foliation \emph{stable} if every sufficiently
small smooth periodic perturbation of the ambient metric admits a minimal foliation
smoothly conjugate to the original one.

Note that not every affine foliation is stable. In the codimension-one case, it is easy to see that a foliation by rational parallel hyperplanes disintegrates into a collection of finitely many minimal hypersurfaces under arbitrarily small perturbation. On the other hand, Moser  \cite{Moser1988} proved that a Diophantine foliation of a flat torus by parallel minimal
hypersurfaces is stable. His theorem belongs to a broader
theory of periodic variational problems developed in \cite{Moser1986}. The
proof replaces the family of leafwise Euler-Lagrange equations by one
degenerate equation on the ambient torus. Its linearization at minimal hypersurfaces is the Jacobi operator. At $g_0$, it becomes the Laplacian along the affine plane, so its Fourier representation involves small
divisors, which are controlled by the Diophantine condition. Moser derived coefficient-dependent Sobolev estimates and carried
out a Newton-smoothing iteration based on his rapidly convergent scheme
\cite{Moser1966I,Moser1966II}. He also observed that an abstract generalized
implicit-function theorem could be used in principle, as in
\cite{Zehnder1975}.

A special higher-codimensional persistence result was subsequently obtained
by Moser for foliations of almost complex tori by pseudoholomorphic curves
\cite{Moser1995}. Those leaves have real dimension two and codimension
$2n-2$ in $\T^{2n}$. As emphasized by Bangert \cite{Bangert1995}, they are homologically
area-minimizing for a metric associated with the almost complex structure. The special Cauchy-Riemann structure of that problem is
not available for general minimal submanifolds. 

At the opposite endpoint, when $m=1$, the leaves are geodesics, and the
stability statement is a classical consequence of KAM theory. A
Diophantine affine geodesic foliation corresponds to an invariant
Lagrangian graph of the geodesic Hamiltonian. The persistence of that graph and
the fact that its projection remains a diffeomorphism produce the nearby
geodesic foliation. The configuration-space formulation of Salamon and
Zehnder is particularly close to the present approach
\cite{SalamonZehnder1989}, see also \cite{Poschel2001}. For a survey of how KAM and Aubry-Mather theory relate to the problems of geodesics on $\mathbb{T}^2$, see the article \cite{Bangert1988Mather} of Bangert.

In codimension one, a transversely oriented minimal hypersurface foliation is
calibrated by contraction of the ambient volume form with its unit normal,
so its leaves are area-minimizing among compactly supported competitors
\cite{HarveyLawson1982}. This minimizing
structure also underlies Bangert's gap phenomenon: sufficiently large
perturbations can create gaps in the family of minimizers and replace a
foliation by a lamination \cite{Bangert1987,Bangert1989}. For more recent complementary results relating minimal foliation to the global geometry of the Riemannian metric, see works of Marques, Neves and Sun \cite{MarquesNevesSun2026} and Nguyen \cite{NguyenStableNorm2026, NguyenMinMax2026}.

There is also an extensive literature on the related realization problem.
Given a foliation, determine whether there exists some ambient metric that makes every leaf
minimal. For oriented foliations, this is one of the standard formulations of
tautness. Fundamental criteria were established by Rummler and Sullivan,
with Sullivan giving a homological characterization in terms of foliation
cycles, and Haefliger later formulated a criterion in terms of transverse
holonomy \cite{Rummler1979,Sullivan1979,Haefliger1980}. In that problem the
foliation is fixed, and we solve for the metric. Here the metric is prescribed
and perturbed, while the foliation moves within its smooth conjugacy class.

In this paper, we prove the corresponding stability theorem for affine foliations of
arbitrary leaf dimension and codimension under general smooth periodic
Riemannian perturbations. The arithmetic hypothesis has the same origin as
in KAM theory: it bounds the inverse leafwise frequencies polynomially and
therefore causes only a fixed finite loss of ambient derivatives. Related
small-divisor mechanisms occur in configuration-space KAM theory, Kozlov's
reducibility theory, and Zehnder's generalized implicit-function theorem
\cite{SalamonZehnder1989,Kozlov1983,Zehnder1975}. This derivative loss is the
reason for working in Hamilton's tame Fr\'echet category rather than in a
single Banach space.

Fix an orthogonal splitting
\[
  \R^N=P\oplus Q,
  \qquad \dim P=m,
  \qquad \dim Q=q=N-m.
\]
The precise Diophantine condition is stated in
\cref{def:diophantine}. Here and below, $C^\infty_0$ denotes the space of smooth maps on the torus with zero average. The main theorem of the paper is the following.

\begin{Th}
\label{thm:main}
Assume that $P$ is Diophantine relative to $\Z^N$. Then there exist an
integer $r_0$, a number $\varepsilon>0$, a $C^\infty$-open neighborhood
$\cG$ of $g_0$ in the space $\cM^\infty$ of smooth
$\Z^N$-periodic Riemannian metrics, and a smooth tame map
\[
  \cG\longrightarrow C^\infty_0(\T^N;Q),
  \qquad g\longmapsto u(g),
\]
with $u(g_0)=0$, such that
\[
  \left\{g\in\cM^\infty:
  \norm[H^{r_0}]{g-g_0}<\varepsilon\right\}\subset\cG.
\]
For every $g\in\cG$, the equivariant map
\[
  \Phi_{u(g)}(z)=z+u(g)(z)
\]
is a diffeomorphism of $\R^N$ commuting with the $\Z^N$-action, and every
leaf
\[
  \Sigma_y(g)=\Phi_{u(g)}(P+y),
  \qquad y\in Q,
\]
is an embedded minimal $m$-submanifold of $(\R^N,g)$. The family is
$\Z^N$-invariant and descends to a smooth minimal foliation of
$(\T^N,g)$ conjugate to the affine foliation. In particular, the affine foliation determined by $P$ is stable, in
Moser's sense, under sufficiently small smooth $\Z^N$-periodic
perturbations of the Euclidean metric.

Moreover, after possibly shrinking $\cG$, there exist an integer
$K_\mathcal{G}\ge0$ and a number $\eta>0$ such that the normalized graph
displacement is locally unique. More precisely, if $g\in\cG$ and
$\widetilde u\in C^\infty_0(\T^N;Q)$ determine a minimal graph
foliation and satisfy
\[
  \norm[H^{K_\mathcal{G}}]{\widetilde u-u(g)}<\eta,
\]
then $\widetilde u=u(g)$.
\end{Th}

The minimal property of the foliation constructed is characterized by having zero mean curvature. In contrast to codimension one, minimal foliations of higher codimension need not admit a common calibration, nor must their leaves be area-minimizing. This failure occurs even for metrics arbitrarily close to the flat metric. For example, on $\mathbb T^3$ the foliation by circles tangent to $\partial_x$ is geodesic for
\[
g_\varepsilon
=(dx+\varepsilon\sin(2\pi y)\,dz)^2+dy^2+dz^2,
\]
but admits no common calibration for $\varepsilon\ne0$. Indeed, the only possible calibrating one-form is $g_\varepsilon(\partial_x,\cdot)$, which is not closed. Moreover, sufficiently long segments of the lifted leaves fail to minimize length. In the next theorem, we show that such behavior is excluded for the Diophantine foliations constructed in \cref{thm:main}.

\begin{Th}
\label{thm:calibrated}
Under the assumptions of \cref{thm:main}, after possibly
shrinking $\cG$, the foliation
$\{\Sigma_y(g):y\in Q\}$ admits a smooth common calibration $\omega_g$
for every $g\in\cG$. Consequently, every compact domain
of a lifted leaf is area-minimizing among compactly
supported integral $m$-currents with the same boundary.
\end{Th}

\begin{Rem}
\label{rem:general-lattice}
Both theorems and their proofs also apply to $\T^N_\Gamma=\R^N/\Gamma$ for any
fixed full lattice $\Gamma$, with the Euclidean reference metric.
One replaces the translation lattice $\Z^N$ by $\Gamma$ and the Fourier
index set by
\[
  \Gamma^*=\{\xi\in(\R^N)^*:\xi(\gamma)\in\Z
                 \text{ for every }\gamma\in\Gamma\},
\]
and normalizes Haar measure to have total mass one. The Diophantine
condition is imposed on the tangential projections of nonzero elements
of $\Gamma^*$, and the constants may depend on the fixed lattice.
The same modifications apply to the variational extension in
\cref{sec:general-integrands}.
\end{Rem}

\begin{Rem}
In \cite[Theorem~1.1]{ChierchiaFalcolini1996}, Chierchia and Falcolini
proved the existence of analytic quasiperiodic solutions for
semilinear elliptic systems
\[
  \Delta_y v=\varepsilon\nabla_x f(v,y),
\]
where $f$ is analytic and periodic, under a Diophantine condition
on the frequency matrix. Their theorem allows arbitrary numbers
of independent variables and components. For sufficiently small
$\varepsilon$, the periodic correction in their parameterization is jointly analytic in the
torus variables and the perturbation parameter. The associated integrand has the restricted form
$\frac12|p|^2+\varepsilon f(x,y)$, with fixed quadratic dependence
on the slope. Their proof establishes convergence of a formal power series
through a tree expansion and cancellations among small-divisor
contributions. In \cite[Remark~1.3]{ChierchiaFalcolini1996}, they note the
limitations of their method for nonanalytic data.
The present theorem works in the $C^\infty$ category, with
smallness measured in one fixed finite Sobolev norm.
The variational extension in \cref{sec:general-integrands} treats
smooth integrands with general slope dependence, including the
Riemannian area integrand near a flat metric.
We do not address the analytic dependence established in
their theorem.
\end{Rem}

\medskip
\noindent\textbf{Sketch of the proof.}
We describe nearby foliations by equivariant maps
$\Phi_u(z)=z+u(z)$, where $u$ is a periodic $Q$-valued
function. The minimality of every leaf is equivalent to
the vanishing of their mean curvatures and is equivalent to the vanishing of an Euler-Lagrange map $\mathcal H(u,g)$
on the ambient torus. The transverse relabelling identity shows that projecting this map to the zero-average space gives the equivalent condition on the vanishing of the projection
$\mathcal P(u,g)=\Pi_0\mathcal H(u,g)$.

We want to apply the Nash-Moser implicit function theorem for $\mathcal{P}$. The main difficulty is the inversion of the linearization. At an exact solution, the linearization is a Jacobi operator on each leaf, and transverse relabeling
generates Jacobi fields given by the columns of
$J=I_Q+\partial_Q u$. The substitution $v=Jw$ sends these Jacobi fields to constant fields and transforms the linearization. In codimension one, the resulting operator is
already in divergence form. However, in higher codimension, a
first-order term with skew-symmetric matrix coefficients
$K^j$ remains. We show that it can be decomposed into a divergence-free part and a part that is
linear in $\mathcal P(u,g)$. The divergence-free part need
not vanish at an exact solution. However, we absorb it into the
second-order coefficient through a matrix-valued potential
$\Theta^{ij}$ and obtain:
\[
  \mathcal{L}_Cw
  =
  -\sum_{i,j=1}^{m}D_i\bigl(C^{ij}D_jw\bigr),
  \qquad
  C^{ij}=M^{ij}+\Theta^{ij}.
\]
The coefficient $C$ is block-symmetric and coercive near $(u,g)=(0,g_0)$.

At $(u,g)=(0,g_0)$, this is just the usual Laplacian along $P$. The Diophantine condition gives a tame inverse for this Laplacian on zero-average fields, with a fixed finite loss of ambient derivatives. Solving for the inverse of $\mathcal{L}_C$ yields an approximate right inverse $V$
satisfying
\[
  D_u\mathcal P(u,g)V(u,g)h
  =
  h+\mathscr Q(u,g)\{\mathcal P(u,g),h\},
\]
where $\mathscr Q$ is smooth tame and bilinear in its last
two arguments. Hamilton's Nash--Moser implicit-function
theorem with quadratic error then gives the desired
foliation and smooth tame dependence on the metric.
Finally, the two-sided inverse at an exact solution,
Taylor's formula, and Sobolev interpolation give local
uniqueness of $u$.
To show the existence of a calibration, we first use the conjugating
diffeomorphism to pull back the metric, then define a canonical characteristic form, which can be made closed by adding a small correction
constructed using the inverse leafwise Laplacian. Pushing the corrected
form forward produces the desired calibration.
\medskip

The paper is organized as follows. \Cref{sec:tame-hamilton} records the tame
Fr\'echet framework and the two results of Hamilton that will be used later.
\Cref{sec:geometry} formulates the graph Euler-Lagrange equation and replaces
its moving compatibility condition by a fixed-range space. In
\cref{sec:linearization}, we study the linearization of the \textit{Euler-Lagrange} map, particularly the decomposition of the skew current $K$. \Cref{sec:inverse-main} proves tame invertibility and smooth tame dependence of the inverse of the transformed operator. In \cref{sec:approx,sec:nashmoser} we construct the
approximate right inverse with quadratic error and prove
\cref{thm:main,thm:calibrated}.
\Cref{sec:general-integrands} records a controlled extension to periodic
graph integrands satisfying the strong ellipticity condition. Sobolev interpolation
and the composition arguments are collected in Appendix~A. Appendix~B proves
the commutator estimate used in \Cref{sec:inverse-main}.
\\

\noindent\textbf{Acknowledgements}. I am deeply grateful to my advisor, Andr\'{e} Neves, for suggesting the research questions that led to this work, and for his many helpful discussions, guidance, and support throughout the project.
\\

\noindent\textbf{AI usage statement}. The author initially proved \cref{thm:main} without AI assistance and only used it for literature searches and polished writing. After completion of the original manuscript, ChatGPT suggested investigating the calibrated property of the foliation, which led to \cref{thm:calibrated}. The author retains exclusive responsibility for the correctness and validity of all the results presented in this paper. 

\section{Hamilton's tame implicit function theorem}
\label{sec:tame-hamilton}

We use Hamilton's tame Fr\'echet calculus in the form developed in
\cite[Parts~II and III]{Hamilton1982}. Only the conventions needed later are
recalled here. All spaces are real, and vector- and matrix-valued norms are
taken componentwise in the fixed Euclidean frames. For further details, see Hamilton's exposition.

A \emph{grading} on a Fr\'echet space \(F\) is an increasing sequence of
seminorms \(\norm[n]{\cdot}\) defining its topology. A linear map
\(T:F\to G\) between graded Fr\'echet spaces is \emph{tame} if there
are integers \(r,b\ge0\) such that
\[
  \norm[G,n]{Tx}
  \le
  C_n\norm[F,n+r]{x},
  \qquad n\ge b.
\]
Following Hamilton, we call \(r\) the degree and \(b\) the base of the
estimate. In the differential-operator setting, \(r\) measures the fixed
loss of derivatives.

For an open set $U\subset F$, a continuous nonlinear map \(\Phi:U\to G\) is \emph{tame} if near every
\(x_0\in U\), there are integers \(r,b\ge0\) and constants \(C_n\) such
that
\[
  \norm[G,n]{\Phi(x)}
  \le
  C_n\bigl(1+\norm[F,n+r]{x}\bigr),
  \qquad \forall n\ge b.
\]
The map is \emph{smooth tame} if it is smooth and for every \(k\ge0\),
the Fr\'echet derivative map
\[
  D^k\Phi:U\times F^k\longrightarrow G
\]
is tame.

For each fixed derivative order, the tameness of the derivative is equivalent to the following multilinear
estimates locally near a chosen base point: for each
$x_0\in U$ and each integer $k\ge0$, there are a neighborhood $U_k$ of
$x_0$, integers $b_k,r_k\ge0$, and constants $C_{n,k}$ such that, for
$n\ge b_k$,
\begin{align}
 &\norm[G,n]{D^k\Phi(x)[h_1,\ldots,h_k]}\notag\\
 &\quad\le C_{n,k}\Bigg[
 \bigl(1+\norm[F,n+r_k]{x}\bigr)
 \prod_{j=1}^k\norm[F,b_k]{h_j}
 +\sum_{\ell=1}^k\norm[F,n+r_k]{h_\ell}
 \prod_{\substack{j=1\\j\ne\ell}}^k\norm[F,b_k]{h_j}
 \Bigg]
\label{eq:uniform-tame-form}
\end{align}
for $x\in U_k$ and all $h_1,\ldots,h_k\in F$. For $k=0$, the empty
product is $1$, and the sum is absent. These homogeneous estimates follow
by rescaling the direction variables in the local tame estimates; see
\cite[Part~II, Lemmas~2.1.7--2.1.8]{Hamilton1982}.

We use the following two theorems of Hamilton. The first upgrades a continuous tame family
of inverses to a smooth tame family.

\begin{Th}[Hamilton's smooth tame inverse-family theorem]
\label{thm:hamilton-inverse-family}
Let $F,H,K$ be tame Fr\'echet spaces, let $U\subset F$ be open, and let
\[
  L:U\times H\longrightarrow K
\]
be smooth tame and linear in the second variable. Suppose that for every
$f\in U$ and $k\in K$ the equation $L(f)h=k$ has a unique solution
$h=V(f)k$, and that $V: U\times K\to H$ is continuous and tame. Then $V$ is
smooth tame.
\end{Th}

This is part~II, Theorem~3.1.1 in \cite{Hamilton1982}. 

\begin{Th}[Implicit function theorem with quadratic error]
\label{thm:hamilton}
Let $F,G,H$ be tame Fr\'echet spaces, $U\subset F \times G$ be an open set, and
\[
  A:U\subset F\times G\longrightarrow H
\]
be a smooth tame map. Suppose there are smooth tame maps
\[
  V:U\times H\longrightarrow F,
  \qquad
  \mathscr Q:U\times H\times H\longrightarrow H,
\]
linear in the last variable of $V$ and bilinear in the last two variables of
$\mathscr Q$, such that
\begin{equation}
  D_fA(f,g)V(f,g)h
  =h+\mathscr Q(f,g)\{A(f,g),h\}.
\label{eq:Hamilton-hyp}
\end{equation}
If $A(f_0,g_0)=0$ for some $(f_0, g_0)\in U$, then after shrinking neighborhoods, there is a smooth
tame map $B$ defined for $g$ near $g_0$ such that
\[
  B(g_0)=f_0,
  \qquad
  A(B(g),g)=0.
\]
\end{Th}

This is part~III, Theorem~3.3.1 of \cite{Hamilton1982}. At a zero of $A$, equation
\eqref{eq:Hamilton-hyp} says that $V$ is an exact right inverse of the
partial derivative $D_fA$.

We now specify the tame Fr\'echet spaces used in our proof.
We identify the dual lattice of $\Z^N$ with $\Z^N$ using the Euclidean
inner product, and use the measure $dz$ on $\T^N$ normalized to have
mass one. For a vector-valued function $f$, consider the Wiener-Sobolev norm:
\begin{equation}
  \norm[H^s]{f}^2
  =
  \sum_{\xi\in\Z^N}
  \langle\xi\rangle^{2s}
  |\widehat f(\xi)|^2,
  \qquad
  \langle\xi\rangle
  =
  \bigl(1+4\pi^2|\xi|^2\bigr)^{1/2}.
\label{eq:Sobolev-grading}
\end{equation}
The corresponding Fr\'echet grading is
\[
  \norm[n]{f}:=\norm[H^n]{f},
  \qquad n=0,1,2,\ldots.
\]

We use Sobolev embedding and the tame product estimate in the following
form, see \cite[\S2, Eqs.~(2.10), (2.18), and~(2.20)]{Moser1988}.
For a fixed integer $s_0>N/2$,
\begin{equation}
  \norm[L^\infty]{f}\le C\norm[H^{s_0}]{f},
\label{eq:sobolev-embedding}
\end{equation}
and, for every integer $s\ge0$,
\begin{equation}
  \norm[H^s]{fg}
  \le C_s\Bigl(
    \norm[H^s]{f}\norm[H^{s_0}]{g}
    +\norm[H^{s_0}]{f}\norm[H^s]{g}
  \Bigr).
\label{eq:tame-product}
\end{equation}
In particular, $H^{s_0}$ is a Banach algebra. These estimates apply
componentwise to vector- and matrix-valued fields on the fixed standard torus.
Sobolev-scale interpolation is recorded in
\cref{lem:sobolev-interpolation}.

Let
\[
  X=Y=C^\infty_0(\T^N;Q),
\]
where the subscript denotes zero average, and define
\[
  \Pi_0f
  =
  f-\int_{\T^N}f\,dz.
\]
The projection $\Pi_0$ is tame linear, so $X$ and $Y$ are tame
Fr\'echet spaces. The space
\[
  G=C^\infty\bigl(\T^N;\Sym^2((\R^N)^*)\bigr)
\]
of smooth symmetric two-tensors is tame, and its open cone of
positive-definite tensors is denoted by $\cM^\infty$.

In our application, $F=X$ and $H=Y$, while the metric parameter
$g$ varies in the open subset $\cM^\infty\subset G$.
The map $A(u,g)=\cP(u,g)$ is defined in \cref{sec:geometry}. Therefore, the proof reduces to constructing smooth tame maps $V$ and
$\mathscr Q$ satisfying \eqref{eq:Hamilton-hyp}.

We also record two composition facts used below. Their proofs are given in \cref{sec:tame-preliminaries}.

\begin{Prop}
\label{prop:composition-evaluation}
Let $V$ be a finite-dimensional vector space.
\begin{enumerate}[label=\textup{(\roman*)}]
\item Fix an integer $s_0>N/2+1$. There is a $C^1$-neighborhood
$\mathcal V$ of zero in $C^\infty(\T^N;\R^N)$ such that
\[
  (F,u)\longmapsto F\circ(\Id+u)
\]
is smooth tame on $C^\infty(\T^N;V)\times\mathcal V$.
On sets with $\norm[H^{s_0}]{u}\le R$, for every integer $s\ge s_0$,
\begin{equation}
\begin{split}
 \norm[H^s]{F\circ(\Id+u)}
 \le C_{s,R}\Bigl(&1+\norm[H^s]{F}\\
 &+\norm[H^s]{u}\bigl(1+\norm[H^{s_0+1}]{F}\bigr)\Bigr),
\end{split}
\label{eq:composition-tame}
\end{equation}
where $\mathcal V$ is chosen with uniform Jacobian bounds.
\item Let $B_{2\rho}\subset\Hom(P,Q)$ be a ball and let
$\mathscr F\in C^\infty(\T^N\times\overline B_{2\rho};V)$
be smooth up to the boundary. For $u\in C^\infty(\T^N;Q)$,
write $D_Pu=du|_P$. On the open set
$\norm[C^0]{D_Pu}<\rho$, the map
\[
  (\mathscr F,u)\longmapsto
  \mathcal E(\mathscr F,u),
  \qquad
  \mathcal E(\mathscr F,u)(z)
  =\mathscr F\bigl(z+u(z),D_Pu(z)\bigr),
\]
is jointly smooth tame in the Sobolev gradings.
\end{enumerate}
\end{Prop}

\section{Geometric formulation and the fixed-range equation}
\label{sec:geometry}
Choose orthonormal bases $e_1,\ldots,e_m$ of $P$ and
$n_1,\ldots,n_q$ of $Q$, and write $D_i=\partial_{e_i},
  \partial_a=\partial_{n_a}$ for the partial derivatives. For a $Q$-valued function $u$, let
\[
  D_Pu:=(D_1u,\ldots,D_m u)
\]
be the gradient of $u$ along $P$. 
The leaf indices $i,j,k$ range from $1$ to $m$, normal indices $a,b,c$ range
from $1$ to $q$, and every sum over such indices is written explicitly.

Using the Euclidean metric, we identify covectors and vectors when convenient.
For $\xi\in\Z^N$, set
\[
  \xi_i=\xi\cdot e_i,\qquad
  \xi_P=(\xi_1,\dots,\xi_m).
\]
The corresponding Fourier character is
\[
  e_\xi(z)=e^{2\pi i\xi\cdot z},
\]
and for every smooth function $f\in C^\infty(\mathbb{T}^N)$ we can write
\[f(z)=\sum_{\xi\in \Z^N}\widehat{f}(\xi)e^{2\pi i \xi\cdot z}, \quad \widehat{f}(\xi)=\int_{\mathbb{T}^N} f(z)e^{-2\pi i\xi\cdot z}\,dz .\]
\begin{Def}
\label{def:diophantine}
The plane $P$ is $(\gamma,\tau)$-Diophantine relative to $\Z^N$ if there are
constants $\gamma>0$ and $\tau\ge0$ such that
\begin{equation}
  |2\pi\xi_P|\ge \gamma\langle\xi\rangle^{-\tau}
  \qquad\text{for every }\xi\in\Z^N\setminus\{0\}.
\label{eq:dioph}
\end{equation}
Since this condition remains valid if we replace $\tau$ by a larger exponent, we can assume $\tau\in \mathbb{N}_{>0}$ for simplicity. 
\end{Def}
For $\tau>q/m$, for almost all $P$ there exists a constant $\gamma$ such that $P$ is $(\gamma, \tau)$-Diophantine. See for example \cite{Sprindzuk1979}.
\begin{Rem}
\label{rem:matrix-diophantine}
When
\[
  P=P_A=\{(x,Ax):x\in\R^m\},
  \qquad A\in\R^{q\times m},
\]
the leafwise vector fields in product coordinates $(x,\theta)$ are
\[
  D_i^A=\partial_{x_i}+\sum_{a=1}^q A_{ai}\partial_{\theta_a}.
\]
Their Fourier symbol on the mode
$e^{2\pi i(k\cdot x+\ell\cdot\theta)}$ is
$2\pi i(k+A^T\ell)$.  Hence \eqref{eq:dioph} is equivalent, up to fixed
norm constants, to the matrix Diophantine condition
\[
  |k+A^T\ell|
  \ge \gamma'\langle(k,\ell)\rangle^{-\tau},
  \qquad (k,\ell)\in\Z^m\times\Z^q\setminus\{0\}.
\]
\end{Rem}

\begin{Rem}
\label{rem:dense-leaves}
The Diophantine condition implies that the image of every affine leaf
$P+y$ is dense in $\T^N$. Indeed, no nonzero $\xi\in\Z^N$
is orthogonal to $P$, so the closure of the image of $P$ has trivial
annihilator and is the whole torus.
\end{Rem}

\subsection{Graphical perturbations of the affine foliation}
\label{sec:euler}

For a periodic map $u:\T^N\to Q$, viewed as a $\Z^N$-periodic map on
$\R^N$, define
\begin{equation}
  \Phi_u(z)=z+u(z).
\label{eq:Phi}
\end{equation}
Then
\[
  \Phi_u(z+k)=\Phi_u(z)+k,
  \qquad \forall k\in\Z^N.
\]
Relative to the fixed splitting $P\oplus Q$, we have
\begin{equation}
 D\Phi_u=
 \begin{pmatrix}
 I_P&0\\
 \partial_Pu&I_Q+\partial_Qu
 \end{pmatrix}.
\label{eq:block-DPhi}
\end{equation}
Hence $\det D\Phi_u=\det(I_Q+\partial_Qu)$.  If $\norm[C^1]{u}$ is sufficiently small,
$D\Phi_u$ is invertible everywhere.  The induced map of the compact torus is
a local diffeomorphism homotopic to the identity, hence a covering of degree
one and therefore a diffeomorphism.  Its lift is an equivariant
diffeomorphism of $\R^N$.

The map $\Phi_u$ maps each affine translate $P+y$ to the submanifold
\begin{equation}
  \Sigma_y(u)=\Phi_u(P+y).
\label{eq:leaves}
\end{equation}
The collection $\{\Sigma_y(u), y\in Q\}$ forms a foliation of $\R^N$. For $z\in P+y$, a basis of $T_{\Phi_u(z)}\Sigma_y(u)$ is given by:
\begin{equation}
  X_i(z)=d\Phi_u(z)[e_i]=e_i+D_i u(z),\qquad 1\le i\le m.
\label{eq:tangent}
\end{equation}

\subsection{Area density and the Euler-Lagrange map}

Let
\[
  \mathcal U\subset C^\infty(\T^N;Q)
\]
be a sufficiently small open neighborhood of the zero map such that
$\Phi_u$ is an equivariant diffeomorphism for every $u\in\mathcal U$.
For $z\in\T^N$, $r\in Q$, and
\[
  p=(p_1,\dots,p_m)\in\Hom(P,Q), p_i\in Q,
\]
define the leaf area density
\[
 L_g:\T^N\times Q\times\Hom(P,Q)\longrightarrow(0,\infty)
\]
by
\begin{equation}
  L_g(z,r,p)=
  \sqrt{\det\bigl(g_{z+r}(e_i+p_i,e_j+p_j)\bigr)_{i,j=1}^m}.
\label{eq:Lg}
\end{equation}

Here $z\in \mathbb{R}^N$ is the reference point, $r$ is the transverse displacement, and
$p$ specifies an $m$-plane transverse to $Q$ as the graph of $p$.  The
expression is interpreted on the universal cover and is independent of the
chosen lift because $g$ is $\Z^N$-periodic.

The pullback of the $g$-area density of the leaf through $\Phi_u(z)$ is given by
\begin{equation}
 \mu_{u,g}(z):=L_g(z,u(z),D_Pu(z)).
\label{eq:area-density}
\end{equation}

Define the averaged graph-area functional
\[
  \mathcal A:\mathcal U\times\cM^\infty\longrightarrow\R
\]
by
\begin{equation}
  \mathcal A(u,g)
  =\int_{\T^N}L_g(z,u,D_Pu)\,dz.
\label{eq:averaged-action}
\end{equation}

The first variational formula for $\mathcal A$ is as follows.
\begin{Lemma}
\label{lem:euler-first-variation}
For every $(u,g)\in\mathcal U\times\cM^\infty$ and every smooth
periodic $v:\T^N\to Q$,
\begin{equation}
  D_u\mathcal A(u,g)[v]
  =\int_{\T^N}
    \bigl\langle\cH(u,g),v\bigr\rangle_0\,dz,
\label{eq:first-variation-global}
\end{equation}
where the map
\[
  \cH:\mathcal U\times\cM^\infty
  \longrightarrow C^\infty(\T^N;Q)
\]
is defined componentwise by
\begin{equation}
  \cH(u,g)_a
  =(L_g)_{r^a}(z,u,D_Pu)
  -\sum_{i=1}^mD_i\!\left((L_g)_{p_i^a}(z,u,D_Pu)\right),
\label{eq:H-definition}
\end{equation}
and $\langle,\rangle_0$ is the standard Euclidean inner product. The same formula holds on a relatively compact domain
$\Omega\Subset P+y$ for compactly supported $Q$-valued graph variations:
\[
  \left.\frac{d}{dt}\right|_{t=0}
  \operatorname{Area}_g\bigl(\Phi_{u+tv}(\Omega)\bigr)
  =\int_\Omega\langle\cH(u,g),v\rangle_0\,d\mathcal H^m_0.
\]
\end{Lemma}

\begin{proof}
Differentiating under the integral gives
\[
  \left.\frac d{dt}\right|_{t=0}
  L_g(z,u+tv,D_Pu+tD_Pv)
  =\sum_{a=1}^q(L_g)_{r^a}v^a
   +\sum_{i=1}^m\sum_{a=1}^q(L_g)_{p_i^a}D_iv^a.
\]
Periodic integration by parts yields \eqref{eq:first-variation-global} and
\eqref{eq:H-definition}. The same calculation on $\Omega$, with $v$
compactly supported, proves the local formula.
\end{proof}

We call $\cH$ the \emph{Euler-Lagrange map}.  Its precise relation to the geometric
mean-curvature vector is given in \cref{prop:H-mean-curvature}. We next show that this map is smooth tame.

\begin{Lemma}
\label{lem:H-smooth-tame}
The Euler-Lagrange map
\[
   \cH:\mathcal U\times\cM^\infty
  \longrightarrow C^\infty(\T^N;Q), \quad (u,g)\longmapsto\cH(u,g), 
\]
is smooth tame. 
\end{Lemma}

\begin{proof}
The proof essentially follows \cite[Part~II, Corollary~2.2.7]{Hamilton1982}, as for each fixed $g$, $\mathcal{H}$ is a nonlinear differential operator of order $2$. To verify joint smooth tameness in $(u,g)$, set
\[
  G_{ij}=g_{\Phi_u}(X_i,X_j),
\]
and write \(G^{ij}\) for the entries of \(G^{-1}\). Jacobi's formula for the derivative of the determinant gives
\[
  (L_g)_{r^a}(z,u,D_Pu)
  =
  \frac{\sqrt{\det G}}{2}
  \sum_{k,l=1}^m
  G^{kl}
  (\partial_a g)_{\Phi_u}(X_k,X_l),
\]
and
\begin{equation}\label{areaform}
  (L_g)_{p_i^a}(z,u,D_Pu)
  =
  \sqrt{\det G}
  \sum_{j=1}^m
  G^{ij}g_{\Phi_u}(n_a,X_j).
\end{equation}
Thus the first expression is a smooth function of $u, D_Pu, g\circ\Phi_u$ and $(\partial_a g)\circ\Phi_u$, whereas the second depends smoothly on $u, D_Pu$ and $g\circ\Phi_u$. 

By \cref{prop:composition-evaluation}\textup{(i)}, the product estimate
\eqref{eq:tame-product}, and \cref{lem:finite-substitution},
for every integer \(s\ge0\),
\[
  \norm[H^s]{(L_g)_r(z,u,D_Pu)}
  +
  \norm[H^{s+1}]{(L_g)_p(z,u,D_Pu)}
  \le
  C_s\left(
    1+\norm[H^{s+2}]{u}
     +\norm[H^{s+1}]{g}
  \right).
\]
Here \(C_s\) depends only on fixed low-norm bounds and the uniform
positivity constants of the neighborhood.

Using
\[
  \cH_a(u,g)
  =
  (L_g)_{r^a}(z,u,D_Pu)
  -
  \sum_{i=1}^m
  D_i\!\left(
    (L_g)_{p_i^a}(z,u,D_Pu)
  \right),
\]
we obtain
\[
  \norm[H^s]{\cH(u,g)}
  \le
  C_s\left(
    1+\norm[H^{s+2}]{u}
     +\norm[H^{s+1}]{g}
  \right).
\]
Finally, every Fr\'echet derivative with respect to \((u,g)\) is a finite
sum of compositions and products of the same type. The tame product
estimate places at most one factor in the highest Sobolev norm, while all
remaining factors are controlled in a fixed norm. It then follows from induction that all Fr\'echet
derivatives satisfy the corresponding multilinear tame estimates, and
\(\cH\) is smooth tame.
\end{proof}

Let
$\mathbf H_{u,g}$ be the mean-curvature vector of $\Sigma_y(u)$, with the sign convention
\[
 \left.\frac d{dt}\right|_{0}\operatorname{Area}_g(\Sigma_t)
 =-\int_\Sigma g(\mathbf H,V)\,d\mu_\Sigma
\]
for compactly supported variation. 
\begin{Prop}
\label{prop:H-mean-curvature}
For $z\in P+y$ and every $\zeta\in Q$, we have
\begin{equation}
  \bigl\langle\cH(u,g)(z),\zeta\bigr\rangle_0
  =-\mu_{u,g}(z)\,
  g_{\Phi_u(z)}\bigl(\mathbf H_{u,g}(\Phi_u(z)),\zeta\bigr).
\label{eq:H-mean-curvature-exact}
\end{equation}
Equivalently, if
\[
  B_{u,g}(z):N^g_{\Phi_u(z)}\Sigma_y(u)\longrightarrow Q^*,
  \qquad
  B_{u,g}(H)[\zeta]=g_{\Phi_u(z)}(H,\zeta),
\]
then
\begin{equation}
  \cH(u,g)(z)
  =-\mu_{u,g}(z)B_{u,g}(z)\mathbf H_{u,g}(\Phi_u(z)).
\label{eq:H-BH}
\end{equation}
The map $B_{u,g}(z)$ is an isomorphism at every graph point. Consequently,
\[
  \cH(u,g)=0
\]
if and only if every leaf \(\Sigma_y(u)\) is minimal.
\end{Prop}

\begin{proof}
Let $\Omega\Subset P+y$ and let $v$ be compactly supported in $\Omega$.
The variation $u_t=u+tv$ has ambient variation field
$V(\Phi_u(z))=v(z)$. By the local formula in
\cref{lem:euler-first-variation},
\[
  \delta\operatorname{Area}_g
  =\int_\Omega\langle\cH(u,g),v\rangle_0\,d\mathcal H^m_0.
\]
The geometric first-variation formula and
$d\mu_{\Sigma_y(u)}=\mu_{u,g}\,d\mathcal H^m_0$ give
\[
  \delta\operatorname{Area}_g
  =-\int_\Omega\mu_{u,g}\,
  g(\mathbf H_{u,g},v)\,d\mathcal H^m_0.
\]
Since the value of $v(z)\in Q$ is arbitrary, the pointwise identity follows.

To prove that $B_{u,g}(z)$ is an isomorphism, suppose that
$H\in N^g_{\Phi_u(z)}\Sigma_y(u)$ and $B_{u,g}(H)=0$. Then $H$ is
$g$-orthogonal both to $T_{\Phi_u(z)}\Sigma_y(u)$ and to $Q$. Because a
graph satisfies
\[
  T_{\Phi_u(z)}\Sigma_y(u)\oplus Q=\R^N,
\]
we obtain $H=0$. The domain and range both have dimension $q$, so $B_{u,g}$
is an isomorphism.
\end{proof}

We use the Euler-Lagrange map rather than the geometric mean curvature
vector because $\cH(u,g)$ takes values in the fixed space $Q$, whereas the
mean curvature vector is a section of a normal bundle that varies with
$(u,g)$.

\subsection{Transverse relabelling and a fixed codomain}
\label{sec:gauge}
For $c\in Q$ define
\begin{equation}
  (T_cu)(z)=c+u(z+c).
\label{eq:gauge-action}
\end{equation}
Then
\[
  \Phi_{T_cu}(z)=\Phi_u(z+c),
\]
so $u$ and $T_cu$ describe the same unlabelled foliation, with the transverse
parameter shifted by $c$. Introduce the \textit{transverse Jacobian}
\begin{equation}
  J=I_Q+\partial_Q u.
\label{eq:J}
\end{equation}
Its $a$-th column is given by $J_a=n_a+\partial_a u$.  Since $u$ is periodic,
\begin{equation}
  \int_{\T^N}J\,dz=I_Q.
\label{eq:meanJ}
\end{equation}

\begin{Lemma}
\label{lem:noether}
For every $(u,g)\in\mathcal U\times\cM^\infty$,
\begin{equation}
  \int_{\T^N}J^T\cH(u,g)\,dz=0.
\label{eq:noether}
\end{equation}
\end{Lemma}

\begin{proof}
The change of variables $z'=z+c$ and the identity
\[
  L_g(z,T_cu,D_P(T_cu))=L_g(z+c,u(z+c),D_Pu(z+c))
\]
show that $\mathcal A(T_cu,g)=\mathcal A(u,g)$.  Note that
\[\left.\frac{d}{dt}\right|_{t=0}T_{tc}u(z)=J(z)c.\]
Differentiating using the fact that $\mathcal A(T_{tc}u,g)$ is constant along $t$, we have
\[0=\left.\frac{d}{dt}\right|_{t=0}\mathcal A(T_{tc}u,g)=D_u\mathcal{A}(u,g)[Jc]=\int_{\T^N}
  \left\langle
    \mathcal H(u,g),Jc
  \right\rangle_0\,dz \]
 The last term is equal to $\int_{\T^N}
  \left\langle
    J^T\mathcal H(u,g),c
  \right\rangle_0\,dz$. Since $c\in Q$ is arbitrary, the desired identity follows. 
\end{proof}
Thus, for each $u$, the image of the Euler-Lagrange map lies in the
codimension-$q$ subspace
\[
  \mathcal Y_u
  =\left\{f\in C^\infty(\T^N;Q):
    \int_{\T^N}J^Tf\,dz=0\right\}.
\]
Equivalently, $\mathcal Y_u$ is the $L^2$-orthogonal complement of the
infinitesimal relabelling directions $\{Jc:c\in Q\}$. Since $J$ depends
on $u$, this is a moving codomain. We replace it by an equivalent equation
with values in the fixed space $Y$.

Recall that
\[
  \Pi_0:C^\infty(\T^N;Q)\longrightarrow
  Y=C^\infty_0(\T^N;Q),
  \qquad
  \Pi_0f=f-\int_{\T^N}f\,dz.
\]
Define the \emph{fixed-range} Euler-Lagrange map
\[
  \cP:\mathcal U\times\cM^\infty\longrightarrow Y
\]
by
\begin{equation}
  \cP(u,g)=\Pi_0\cH(u,g).
\label{eq:Pfixed}
\end{equation}

\begin{Lemma}
\label{lem:fixed-target}
For every $(u,g)\in\mathcal U\times\cM^\infty$,
\begin{equation}
  \cH(u,g)
  =\cP(u,g)
  -\int_{\T^N}J^T\cP(u,g)\,dz.
\label{eq:HfromP}
\end{equation}
Consequently,
\[
  \cP(u,g)=0\quad\Longleftrightarrow\quad\cH(u,g)=0.
\]
\end{Lemma}

\begin{proof}
Write $\cH(u,g)=\cP(u,g)+c$ with $c\in Q$ constant.  Insert this into
\eqref{eq:noether} and use \eqref{eq:meanJ}:
\[
  0=\int J^T\cP(u,g)\,dz+c.
\]
Solving for $c$ yields \eqref{eq:HfromP}.
\end{proof}

\section{Linearization and decomposition of the skew current}
\label{sec:linearization}
In this section, we analyze the linearization of $\cH$ and reduce
$D_u\cP$ to a block-symmetric leafwise divergence-form system. At minimal leaves, this linearization can be considered as the Jacobi operator of the leaves. This
will be used in \cref{sec:approx} to construct the approximate right inverse
required by Hamilton's theorem. All derivatives of $L=L_g$ below are
evaluated at $(z,u(z),D_Pu(z))$. Define $q\times q$ matrices
\begin{equation}
  A^{ij}_{ab}=L_{p_i^a p_j^b},
  \qquad
  B^i_{ab}=L_{p_i^a r^b},
  \qquad
  R_{ab}=L_{r^a r^b}.
\label{eq:ABR}
\end{equation}
The equality of mixed derivatives gives
\begin{equation}
  (A^{ij})^T=A^{ji},
  \qquad R^T=R.
\label{eq:ABRsym}
\end{equation}
The linearization of the Euler-Lagrange map is given by the following.
\begin{Lemma}
\label{lem:linearized-euler}
For a $Q$-valued variation $v\in C^\infty(\mathbb T^N; Q)$ we have
\begin{equation}
  D_u\cH(u,g)[v]
  =-\sum_{i=1}^m D_i\!\left(
      \sum_{j=1}^m A^{ij}D_jv+B^iv
    \right)
   +\sum_{j=1}^m(B^j)^TD_jv+Rv.
\label{eq:linearized-H}
\end{equation}
The operator is formally self-adjoint on the torus.
\end{Lemma}

\begin{proof}
Differentiate \eqref{eq:H-definition}.  The derivative of the first term is
\[
  L_{r r}v+\sum_{j=1}^m L_{r p_j}D_jv
  =Rv+\sum_{j=1}^m(B^j)^TD_jv,
\]
while the derivative of the second term is
\[
  \sum_{i=1}^m D_i\!\left(
    L_{p_i r}v+\sum_{j=1}^m L_{p_i p_j}D_jv
  \right)
  =\sum_{i=1}^m D_i\!\left(
    B^iv+\sum_{j=1}^m A^{ij}D_jv
  \right).
\]
This proves \eqref{eq:linearized-H}.  Formal self-adjointness follows from \eqref{eq:ABRsym} and periodic integration by parts.
\end{proof}

\begin{Lemma}
\label{lem:H-equivariance}
For every $(u,g)$ in the graph-metric neighborhood and every $c\in Q$ for which $T_cu$ remains in the graph neighborhood,
\begin{equation}
  \cH(T_cu,g)(z)=\cH(u,g)(z+c).
\label{eq:Hequivariance}
\end{equation}
Consequently,
\begin{equation}
  D_u\cH(u,g)[J_a]=\partial_a\cH(u,g),
  \qquad a=1,\ldots,q.
\label{eq:Jacobi-columns}
\end{equation}
\end{Lemma}

\begin{proof}
Since
\[
  z+T_cu(z)=z+c+u(z+c),
  \qquad
  D_PT_cu(z)=D_Pu(z+c),
\]
the defining formula \eqref{eq:H-definition} gives
$\cH(T_cu,g)(z)=\cH(u,g)(z+c)$. Differentiating at $c=0$ in the
direction $n_a$ yields \eqref{eq:Jacobi-columns}.
\end{proof}

\subsection{The matrix ground-state substitution}
\label{sec:ground-state-motivation}
To construct an approximate right inverse for $D_u\cP$, we first apply a
matrix ground-state transform to the linearization of $\cH$. The reason is that, at an exact
solution, \eqref{eq:Jacobi-columns} gives
\[
  D_u\cH(u,g)[J_a]=0,
\]
so the columns of $J$ are Jacobi fields generated by transverse relabelling.
In codimension one, this is Moser's positive Jacobi field
\cite{Moser1988}. Therefore, we set
\begin{equation}
  v=Jw.
\label{eq:ground-state-substitution}
\end{equation}
For $u$ sufficiently small, $J$ is pointwise invertible. At an exact solution, this change of variable sends the Jacobi fields $\{Jc: c\in Q\}$ to constant $Q$-valued field.

\begin{Lemma}
\label{lem:ground-state}
For every smooth $Q$-valued function $w$,
\begin{equation}
\begin{aligned}
  J^T\Bigl(
  D_u\cH(u,g)[Jw]
  -\sum_{a=1}^q(\partial_a\cH(u,g))w^a
  \Bigr)
  ={}&-\sum_{i=1}^m D_i\!\left(
      \sum_{j=1}^m M^{ij}D_jw
    \right)\\
   &-\sum_{j=1}^m K^jD_jw,
\end{aligned}
\label{eq:ground-state}
\end{equation}
where
\begin{equation}
  M^{ij}=J^TA^{ij}J
\label{eq:M}
\end{equation}
and $K=(K^1, K^2, \ldots, K^m)$ is the matrix field consisting of $q\times q$ matrices of the form:
\begin{equation}
  K^j
  =\sum_{i=1}^m\left(
      J^TA^{ji}D_iJ-(D_iJ)^TA^{ij}J
    \right)
   +J^T\bigl(B^j-(B^j)^T\bigr)J.
\label{eq:K}
\end{equation}
Moreover,
\begin{equation}
  (M^{ij})^T=M^{ji},
  \qquad
  (K^j)^T=-K^j.
\label{eq:MKsym}
\end{equation}
\end{Lemma}

\begin{proof}
Write $J_b$ for the $b$-th column of $J$.  For $v=Jw$, we have
\begin{equation}
 D_jv=(D_jJ)w+JD_jw.
\label{eq:D-Jw}
\end{equation}
Substitution into \eqref{eq:linearized-H} gives
\begin{align*}
 D_u\cH(u,g)[Jw]
={}&-\sum_{i=1}^m D_i\!\left(
      \sum_{j=1}^m A^{ij}(D_jJ)w+B^iJw
    \right)\\
 &+\sum_{j=1}^m(B^j)^T(D_jJ)w+RJw\\
 &-\sum_{i=1}^m D_i\!\left(
      \sum_{j=1}^m A^{ij}JD_jw
    \right)
 +\sum_{j=1}^m(B^j)^TJD_jw.
\end{align*}
After expanding the first divergence, the terms with no derivative
on $w$ are exactly
$\sum_{a=1}^q D_u\cH(u,g)[J_a]w^a$.  By
\eqref{eq:Jacobi-columns} this equals
$\sum_{a=1}^q(\partial_a\cH(u,g))w^a$.  After subtracting these terms, the remaining
expression is
\begin{equation}
 -\sum_{i=1}^m D_i\!\left(
    \sum_{j=1}^m A^{ij}JD_jw
  \right)
 -\sum_{i,j=1}^m A^{ji}(D_iJ)D_jw
 -\sum_{j=1}^m B^jJD_jw
 +\sum_{j=1}^m(B^j)^TJD_jw.
\label{eq:remaining-ground-rigorous}
\end{equation}
Multiplying by $J^T$ and applying the product rule to the first term yields
\[
\begin{aligned}
 -\sum_{i=1}^m J^TD_i\!\left(
    \sum_{j=1}^m A^{ij}JD_jw
  \right)
 ={}&-\sum_{i=1}^m D_i\!\left(
    \sum_{j=1}^m J^TA^{ij}JD_jw
  \right)\\
 &+\sum_{i,j=1}^m(D_iJ)^TA^{ij}JD_jw.
\end{aligned}
\]
Thus the coefficient of $D_jw$ outside the divergence is
\[
 \sum_{i=1}^m\left(
   (D_iJ)^TA^{ij}J-J^TA^{ji}D_iJ
 \right)
 -J^TB^jJ+J^T(B^j)^TJ=-K^j,
\]
which proves \eqref{eq:ground-state} and \eqref{eq:K}.
Finally,
\[
 (M^{ij})^T=J^T(A^{ij})^TJ=J^TA^{ji}J=M^{ji}.
\]
Transposing \eqref{eq:K} interchanges the first two terms with a minus sign,
while $B^j-(B^j)^T$ changes sign, hence $(K^j)^T=-K^j$.
\end{proof}

\begin{Rem}
We call $K=(K^1,\ldots,K^m)$ the \emph{skew current}. For $q=1$, skew-symmetry forces $K^j=0$, recovering Moser's scalar
cancellation. In higher codimension, the current need not vanish.
\end{Rem}

\subsection{Normal-curl and divergence identities}
\label{sec:wronskian}
We now derive the structural identities satisfied by the skew current $K$. For every transverse direction $n_a$, define the quantity:
\begin{equation}
  \rho_{i,a}
  =L_{p_i}(z,u,D_Pu)\cdot J_a.
\label{eq:rho}
\end{equation}
The next lemma contains all of the
structural identities for $K$ that will be used later.

\begin{Lemma}
\label{lem:wronskian-div}
For every $i,a,b$, we have
\begin{equation}
  (K^i)_{ab}
  =\partial_b\rho_{i,a}-\partial_a\rho_{i,b}.
\label{eq:normal-curl}
\end{equation}
Consequently,
\begin{equation}
  \int_{\T^N}K^i\,dz=0
  \qquad\text{for every }i,
\label{eq:Kzero}
\end{equation}
and the leafwise divergence satisfies
\begin{equation}
  \sum_{i=1}^m D_i(K^i)_{ab}
  =-J_a\cdot\partial_b\cH(u,g)
   +J_b\cdot\partial_a\cH(u,g).
\label{eq:Kdiv}
\end{equation}
In particular, at every exact solution of the Euler-Lagrange equation,
\begin{equation}
  \sum_{i=1}^m D_iK^i=0.
\label{eq:K-div-free-solution}
\end{equation}
\end{Lemma}

\begin{proof}
We first differentiate the density
$L_g(z,u(z),D_Pu(z))$ in the fixed transverse direction $n_a$.  By \eqref{eq:Lg}, the variables $z$ and $r$ enter only through $z+r$.
Therefore, $L_{z_a}=L_{r^a}$ in every transverse direction.  Since
$J_a=n_a+\partial_a u$ and
$\partial_aD_iu=D_i\partial_a u=D_iJ_a$, the chain rule gives
\[
  \partial_a\!\left[L_g(z,u,D_Pu)\right]
  =L_r\cdot J_a
   +\sum_{i=1}^m L_{p_i}\cdot D_iJ_a.
\]
Using
\[
  \cH(u,g)=L_r-\sum_{i=1}^m D_iL_{p_i}
\]
and the definition \eqref{eq:rho}, the product rule yields
\begin{equation}
  \sum_{i=1}^m D_i\rho_{i,a}
  =\partial_a\!\left[L_g(z,u,D_Pu)\right]
   -\cH(u,g)\cdot J_a.
\label{eq:local-balance}
\end{equation}

We next identify the skew coefficient in the matrix ground-state transform.
The chain rule and the fact that the $(z,r)$-dependence of $L$ is only
through $z+r$ give
\begin{equation}
  \partial_b\!\left[L_{p_i}(z,u,D_Pu)\right]
  =\sum_{j=1}^m A^{ij}D_jJ_b+B^iJ_b.
\label{eq:derivative-momentum}
\end{equation}
Interchanging $b$ and $a$ then subtracting gives
\begin{align*}
 \partial_b\rho_{i,a}-\partial_a\rho_{i,b}
={}&J_a^T\partial_bL_{p_i}-J_b^T\partial_aL_{p_i}
 +L_{p_i}\cdot(\partial_bJ_a-\partial_aJ_b).
\end{align*}
Here and below all transverse derivatives are evaluated at
$(z,u,D_Pu)$, we drop the base point for computational convenience.  The last term vanishes because the transverse derivatives
commute.  Substituting \eqref{eq:derivative-momentum} gives
\[
 \partial_b\rho_{i,a}-\partial_a\rho_{i,b}=\sum_{j=1}^m\left(
   J_a^TA^{ij}D_jJ_b-J_b^TA^{ij}D_jJ_a
 \right)
 +J_a^TB^iJ_b-J_b^TB^iJ_a.
\]
Using $(A^{ij})^T=A^{ji}$, this is exactly the $(a,b)$ entry of $K^i$ in
\eqref{eq:K}, proving \eqref{eq:normal-curl}.  Since every $\rho_{i,a}$ is
periodic, the transverse derivatives in \eqref{eq:normal-curl} have zero torus
average, and \eqref{eq:Kzero} follows.

Finally, take the leafwise divergence of \eqref{eq:normal-curl}.  Since all
$D_i$ and $\partial_a$ are constant commuting vector fields, equation
\eqref{eq:local-balance} yields
\begin{align*}
 \sum_{i=1}^mD_i(K^i)_{ab}
 &=\partial_b\partial_aL-
   \partial_b\bigl(\cH(u,g)\cdot J_a\bigr)
   -\partial_a\partial_bL+
   \partial_a\bigl(\cH(u,g)\cdot J_b\bigr)\\
 &=-J_a\cdot\partial_b\cH(u,g)
   +J_b\cdot\partial_a\cH(u,g),
\end{align*}
where the mixed transverse derivatives of the composed density and of $u$ cancel. This proves
\eqref{eq:Kdiv}.  If $\cH(u,g)=0$, then all of its transverse derivatives
vanish, giving \eqref{eq:K-div-free-solution}.
\end{proof}

\subsection{Decomposition of the skew current}
\label{sec:hodge-reduction}
\Cref{lem:wronskian-div} determines only the leafwise divergence of $K$, but not the whole term.  This distinction is essential for Hamilton's theorem: a part
of $K$ may survive when the residual vanishes and therefore cannot be placed
in the quadratic error.  We now decompose $K$ into a divergence-free part, which
belongs to the exact transformed operator, and a remaining part, which is
linear in the residual. 

Define the positive leafwise Laplacian by
\begin{equation}
  \Delta_P=-\sum_{i=1}^m D_i^2.
\label{eq:leaf-laplacian}
\end{equation}

\begin{Lemma}
\label{lem:leafwise-laplacian}
Assume that $P$ satisfies the Diophantine condition \eqref{eq:dioph}.  Then
\begin{equation}
 \Delta_P:C^\infty_0(\T^N;Q)
 \longrightarrow C^\infty_0(\T^N;Q)
\label{eq:leaf-laplacian-map}
\end{equation}
is a bijection.  Its inverse is the Fourier multiplier
\begin{equation}
  \widehat{\Delta_P^{-1}f}(\xi)
  =\frac{\widehat f(\xi)}{|2\pi\xi_P|^2},
  \qquad \xi\ne0,
\label{eq:leaf-inverse}
\end{equation}
with zero Fourier coefficient at $\xi=0$, and for every $s\in\R$ it satisfies
\begin{equation}
  \norm[H^s]{\Delta_P^{-1}f}
  \le C\gamma^{-2}\norm[H^{s+2\tau}]{f}.
\label{eq:leaf-inverse-estimate}
\end{equation}
In particular, $\Delta_P^{-1}$ is a smooth tame linear map with fixed loss
$2\tau$.
\end{Lemma}

\begin{proof}
On the Fourier mode $e_\xi$,
\[
  \Delta_Pe_\xi=|2\pi\xi_P|^2e_\xi.
\]
The Diophantine condition makes this multiplier nonzero for every
$\xi\ne0$. Hence, on zero-average fields, the inverse is given by
\eqref{eq:leaf-inverse}. Moreover,
\begin{align*}
  \norm[H^s]{\Delta_P^{-1}f}^2
  &=\sum_{\xi\ne0}\langle\xi\rangle^{2s}
    |2\pi\xi_P|^{-4}|\widehat f(\xi)|^2\\
  &\le\gamma^{-4}
    \sum_{\xi\ne0}\langle\xi\rangle^{2s+4\tau}
    |\widehat f(\xi)|^2
  =\gamma^{-4}\norm[H^{s+2\tau}]{f}^2.
\end{align*}
This proves the estimate, smoothness, and surjectivity. If
$\Delta_Pw=0$ and $w$ has zero average, then every Fourier coefficient of
$w$ vanishes, proving injectivity.
\end{proof}

For a leaf-indexed matrix field
\[
  R=(R^1,\dots,R^m),
  \qquad R^i\in C^\infty(\T^N;\operatorname{End}(Q)),
\]
define
\begin{equation}
  R_{\parallel}^j
  =-D_j\Delta_P^{-1}\!\left(\sum_{i=1}^mD_iR^i\right),
  \qquad
  \mathbb P_{\df}R=R-R_{\parallel}.
\label{eq:projections}
\end{equation}
At nonzero $\xi\in \Z^N$, we have 
\[
\begin{aligned}
  \widehat{R_\parallel^j}(\xi)=
  -\frac{2\pi i\xi_{P,j}}{|2\pi\xi_P|^2}
  \sum_{i=1}^{m}2\pi i\xi_{P,i}\widehat{R^i}(\xi)
  =
  \sum_{i=1}^{m}
  \frac{\xi_{P,j}\xi_{P,i}}{|\xi_P|^2}
  \widehat{R^i}(\xi).
\end{aligned}
\]
The matrix $\xi_P\xi_P^T/|\xi_P|^2$ is an orthogonal projection in the
leaf-index space, acting separately on each matrix entry. At $\xi=0$,
the definitions give $\widehat{R_\parallel}(0)=0$ and
$\widehat{\mathbb P_{\df}R}(0)=\widehat R(0)$. Hence
\[
  \norm[H^s]{R_\parallel}^2
  +
  \norm[H^s]{\mathbb P_{\df}R}^2
  =
  \norm[H^s]{R}^2,
  \qquad s\in\mathbb R,
\]
so both projections are bounded on every $H^s$ without derivative loss.
At every nonzero Fourier index the coefficients of $\mathbb P_{\df}R$
are orthogonal to $\xi_P$, while derivatives annihilate the zero mode.
Consequently,
\begin{equation}
  \sum_{j=1}^m D_j(\mathbb P_{\df}R)^j=0.
\label{eq:df-div}
\end{equation}
Apply this decomposition to $K$:
\begin{equation}
  K_{\df}=\mathbb P_{\df}K,
  \qquad
  K_{\parallel}=K-K_{\df}.
\label{eq:Kdecomp}
\end{equation}
We now analyze each term $K_{\df}$ and $K_{\parallel}$ separately. 

\begin{Lemma}
\label{lem:Kparallel}
There is a smooth tame map
\[
  T:\mathcal U\times Y\longrightarrow
  C^\infty_0(\T^N;\so(q))^m,
\]
such that, for every fixed $u$, the map $r\mapsto T_u(r)$ is linear.  It is
given componentwise by
\begin{equation}
  \bigl(T_u(r)^j\bigr)_{ab}
  =D_j\Delta_P^{-1}
  \Bigl(J_a\cdot\partial_b r
       -J_b\cdot\partial_a r\Bigr),
\label{eq:Kparallel-formula}
\end{equation}
where $J_a=n_a+\partial_a u$.  Moreover,
\begin{equation}
  \sum_{j=1}^mD_j\bigl(T_u(r)^j\bigr)_{ab}
  =-J_a\cdot\partial_b r
   +J_b\cdot\partial_a r,
\label{eq:Tu-divergence}
\end{equation}
and the component $K_{\parallel}$ of $K$ satisfies
\begin{equation}
  K_{\parallel}=T_u(\cP(u,g)).
\label{eq:Kparallel-residual}
\end{equation}
More precisely, if $b_0>N/2$ is fixed, then for every $s\ge b_0$,
\begin{align}
  \norm[H^s]{T_u(r)}
  \le C_s\gamma^{-1}\Bigl[
  &(1+\norm[H^{b_0+1}]{u})
    \norm[H^{s+\tau+1}]{r}
  \notag\\
  &+\norm[H^{s+\tau+1}]{u}
    \norm[H^{b_0+1}]{r}
  \Bigr].
\label{eq:Tu-tame-estimate}
\end{align}
Thus $T_u$ has a fixed loss of at most $\tau+1$ derivatives.
\end{Lemma}

\begin{proof}
For $r\in Y$, set
\[
  G_{ab}=J_a\cdot\partial_b r-J_b\cdot\partial_a r.
\]
Since $\partial_bJ_a=\partial_aJ_b$, we can write
\[
  G_{ab}=\partial_b(J_a\cdot r)-\partial_a(J_b\cdot r),
\]
so $G_{ab}\in C^\infty_0(\mathbb{T}^N; \mathbb{R})$. Thus the right-hand side of \eqref{eq:Kparallel-formula} is well
defined by \cref{lem:leafwise-laplacian}. It is skew in $a,b$, has zero
average, and is linear in $r$. Moreover,
\[
  \sum_{j=1}^mD_j\bigl(T_u(r)^j\bigr)_{ab}=\sum_{j=1}^{m}D_j(D_j\Delta_P^{-1} G_{ab})=-\Delta_P\Delta_P^{-1} G_{ab}
  =-G_{ab},
\]
which proves \eqref{eq:Tu-divergence}.

By the definition of $K_{\parallel}$ and
\eqref{eq:Kdiv}, we have
\[
  K_\parallel^j
  =-D_j\Delta_P^{-1}\!\left(\sum_{i=1}^mD_iK^i\right)
  =T_u(\cP(u,g))^j,
\]
because $\cH(u,g)-\cP(u,g)$ is spatially constant. This proves
\eqref{eq:Kparallel-residual}.

The multiplier of $D_j\Delta_P^{-1}$ is bounded in magnitude by
$\gamma^{-1}\langle\xi\rangle^\tau$. Hence
\[
  \norm[H^s]{D_j\Delta_P^{-1}F}
  \le C\gamma^{-1}\norm[H^{s+\tau}]{F}
\]
for zero-average $F$. Applying the tame product estimate to
$J_a\cdot\partial_b r$, and using $J=I_Q+\partial_Qu$, gives
\eqref{eq:Tu-tame-estimate}. Since the remaining operations are fixed tame
linear maps, multiplication, and differentiation, it follows that $(u,r)\mapsto T_u(r)$ is
smooth tame.
\end{proof}

\begin{Lemma}
\label{lem:hodge}
Let $R=(R^1,\dots,R^m)$ be a smooth matrix-valued field satisfying
\begin{equation}
  \sum_{i=1}^mD_iR^i=0,
  \qquad
  \int_{\T^N}R^i\,dz=0,
  \qquad
  (R^i)^T=-R^i.
\label{eq:Rconditions}
\end{equation}
Then there are smooth matrix fields $\Theta^{ij}$ such that
\begin{equation}
  \Theta^{ij}=-\Theta^{ji},
  \qquad
  (\Theta^{ij})^T=-\Theta^{ij},
  \qquad
  \sum_{i=1}^mD_i\Theta^{ij}=R^j,
\label{eq:Hconditions}
\end{equation}
and, for every $s\in\R$,
\begin{equation}
  \norm[H^s]{\Theta}
  \le C_s\gamma^{-1}\norm[H^{s+\tau}]{R}.
\label{eq:Hestimate}
\end{equation}
The map $R\mapsto \Theta$ is linear tame.
\end{Lemma}

\begin{proof}
For a nonzero Fourier frequency set
\[
  \eta_i=2\pi(\xi_P)_i,
  \qquad |\eta|^2=|2\pi\xi_P|^2,
\]
and define
\begin{equation}
  \widehat \Theta^{ij}(\xi)
  =\frac{1}{i|\eta|^2}
   \bigl(\eta_i\widehat R^j(\xi)
        -\eta_j\widehat R^i(\xi)\bigr).
\label{eq:Hfourier}
\end{equation}
Set the zero mode to zero.  This Fourier prescription defines a smooth
field.  Indeed, the Diophantine condition gives
\[
  |2\pi\xi_P|^{-1}\le \gamma^{-1}\langle\xi\rangle^\tau,
\]
and therefore
\begin{equation}
  |\widehat \Theta(\xi)|
  \le C|2\pi\xi_P|^{-1}|\widehat R(\xi)|
  \le C\gamma^{-1}\langle\xi\rangle^\tau
       |\widehat R(\xi)|.
\label{eq:Hfourier-bound}
\end{equation}
Since $R$ is smooth, its Fourier coefficients are rapidly decreasing. Therefore, the same is true for the coefficients defined by \eqref{eq:Hfourier}.
Hence the Fourier series for $\Theta$ converges in $C^\infty$.  Squaring
\eqref{eq:Hfourier-bound}, multiplying by $\langle\xi\rangle^{2s}$, and
summing over $\xi$ gives \eqref{eq:Hestimate}. In particular, the map
$R\mapsto \Theta$ is linear tame.

It remains to verify the stated identities.  The divergence-free condition is
$\sum_{i=1}^m\eta_i\widehat R^i=0$, so
\[
  i\sum_{i=1}^m\eta_i\widehat \Theta^{ij}=\widehat R^j.
\]
The two skew symmetries in \eqref{eq:Hconditions} follow directly from \eqref{eq:Hfourier}.
\end{proof}

By \eqref{eq:Kzero}, \eqref{eq:df-div}, and the skew-symmetry of $K$, the field $K_{\df}$ satisfies all hypotheses of \cref{lem:hodge}. Apply that lemma with $R=K_{\df}$ and denote the resulting potential by $\Theta^{ij}$.  Since $D_iD_jw=D_jD_iw$ and $\Theta^{ij}=-\Theta^{ji}$,
\begin{equation}
\begin{split}
  -\sum_{i=1}^mD_i\!\left(
    \sum_{j=1}^m\Theta^{ij}D_jw
  \right)
  &=-\sum_{i,j=1}^m(D_i\Theta^{ij})D_jw
    -\sum_{i,j=1}^m\Theta^{ij}D_iD_jw\\
  &=-\sum_{j=1}^mK_{\df}^jD_jw.
\end{split}
\label{eq:Habsorb}
\end{equation}
Thus the part of the transformed operator that survives on the solution set is
\begin{equation}
  \mathcal{L}_{C(u,g)}w   
  :=-\sum_{i=1}^mD_i\!\left(
      \sum_{j=1}^mM^{ij}D_jw
    \right)
   -\sum_{j=1}^mK_{\df}^jD_jw
\label{eq:Atilde-def}
\end{equation}
and has pure divergence form
\begin{equation}
  \mathcal{L}_{C(u,g)}w
  =-\sum_{i=1}^mD_i\!\left(
      \sum_{j=1}^mC^{ij}D_jw
    \right),
  \qquad
  C^{ij}=M^{ij}+\Theta^{ij}.
\label{eq:Cdef}
\end{equation}
The two skew symmetries of $\Theta$ imply
\begin{equation}
  (\Theta^{ij})^T=-\Theta^{ij}=\Theta^{ji},
\end{equation}
so
\begin{equation}
  (C^{ij})^T=C^{ji}.
\label{eq:Cblocksym}
\end{equation}
The two skew symmetries of $\Theta$ do not force its contribution to
$\sum_{i,j=1}^m\zeta_i^TC^{ij}\zeta_j$ to vanish. Coercivity will follow
instead from closeness of the full coefficient $C$ to its positive
constant value at the flat reference configuration.
Combining the ground-state identity with
$K=K_{\df}+T_u(\cP(u,g))$ gives the final reduced form:
\begin{equation}
\begin{aligned}
  J^T\Bigl(
  D_u\cH(u,g)[Jw]
  -\sum_{a=1}^q(\partial_a\cH(u,g))w^a
  \Bigr)
  ={}&-\sum_{i=1}^mD_i\!\left(
      \sum_{j=1}^mC^{ij}D_jw
    \right)\\
   &-\sum_{j=1}^m(T_u(\cP(u,g)))^jD_jw.
\end{aligned}
\label{eq:reduced-ground-state}
\end{equation}
 The remaining technical task is to establish tame estimates for the
inverse of the block-symmetric divergence-form operator
\[w\mapsto-\sum_{i=1}^mD_i\!\left(\sum_{j=1}^mC^{ij}D_jw\right) \]
and for the residual-dependent first-order term
\[\sum_{j=1}^m(T_u(\mathcal P(u,g)))^jD_jw,\]
which generates the quadratic error in the
Nash-Moser scheme.

\subsection{The geometric coefficient map}
\label{sec:geometric-coefficient}
The preceding reduction produces the block-symmetric coefficient
$C=C(u,g)=M(u,g)+\Theta(u,g)$.  We now identify its value at the flat reference
configuration and record the smooth tame dependence of $C$ on the geometric
data.

At the flat reference configuration,
\[
  L_{g_0}(z,0,p)
  =
  \sqrt{\det(I_m+p^Tp)}
  =
  1+\frac12|p|^2+O(|p|^4).
\]
Consequently,
\[
  \frac{\partial^2 L_{g_0}}
       {\partial p_i^a\partial p_j^b}(z,0,0)
  =
  \delta_{ij}\delta_{ab},
\]
or equivalently,
\[
  A^{ij}(0,g_0)=\delta_{ij}I_Q.
\]
Moreover, at \((u,g)=(0,g_0)\) one has
\[
  J=I_Q,\qquad B^i=0,\qquad K=0,\qquad \Theta=0.
\]
It follows that
\[
  M^{ij}(0,g_0)=\delta_{ij}I_Q
  \quad\text{and}\quad
  C^{ij}(0,g_0)=\delta_{ij}I_Q=:C_0^{ij}.
\]
In particular,
\[
  \sum_{i,j=1}^m\zeta_i^TC_0^{ij}\zeta_j
  =
  \sum_{i=1}^m|\zeta_i|^2
\]
for every \(\zeta\in P^*\otimes Q\).

\begin{Lemma}
\label{lem:coefficient-tame}
The coefficient map
\[
  (u,g)\longmapsto C(u,g)=M(u,g)+\Theta(u,g)
\]
is smooth tame on a neighborhood of $(0,g_0)$. Consequently, for every fixed Sobolev index $s$ and every $\varepsilon>0$,
a sufficiently small $C^\infty$-neighborhood of $(0,g_0)$ satisfies
\begin{equation}
  \norm[H^s]{C(u,g)-C_0}<\varepsilon.
\label{eq:C-local-closeness}
\end{equation}
\end{Lemma}

\begin{proof}
The coefficient $M(u,g)$ is an explicit expression involving finite-order
derivatives and composition with $\Phi_u$. By \cref{prop:composition-evaluation}, \eqref{eq:tame-product}, and
\cref{lem:finite-substitution}, the map $(u,g)\mapsto M(u,g)$ is smooth
tame.  The same reasoning applied to \eqref{eq:K} shows that
$(u,g)\mapsto K(u,g)$ is smooth tame.  Since $\mathbb P_{\df}$ in
\eqref{eq:projections} is a fixed order-zero Fourier multiplier,
$(u,g)\mapsto K_{\df}(u,g)=\mathbb P_{\df}K(u,g)$ is smooth tame as well.
By \cref{lem:hodge}, the map $K_{\df}\mapsto \Theta$ is linear tame. In
particular,
\[
  \norm[H^s]{\Theta(u,g)}
  \le C_s\gamma^{-1}
  \norm[H^{s+\tau}]{K_{\df}(u,g)}.
\]
Therefore $(u,g)\mapsto \Theta(u,g)$ is smooth tame by composition, and hence so
is $C=M+\Theta$.  Finally, \eqref{eq:C-local-closeness}
follows from continuity of the smooth tame map $C$ at $(0,g_0)$.
\end{proof}

\section{Tame inversion of the transformed operator}
\label{sec:inverse-main}

We now study the abstract leafwise divergence-form operator
\begin{equation}
  \mathcal{L}_Cw=-\sum_{i=1}^mD_i\!\left(
    \sum_{j=1}^mC^{ij}D_jw
  \right).
\label{eq:LC}
\end{equation}
Fix a constant block-symmetric coefficient $C_0=(C_0^{ij})$ satisfying
\begin{equation}
  (C_0^{ij})^T=C_0^{ji},
  \qquad
  \sum_{i,j=1}^m\zeta_i^TC_0^{ij}\zeta_j
  \ge2\lambda\sum_{i=1}^m|\zeta_i|^2
\label{eq:C0-coercive}
\end{equation}
for some $\lambda>0$ and every
$(\zeta_1,\ldots,\zeta_m)\in P^*\otimes Q$. The Riemannian application has
$C_0^{ij}=\delta_{ij}I_Q$. We prove a uniform tame inverse theorem for
coefficients close to $C_0$ and then establish smooth tame dependence of the
inverse on $C$.

\subsection{The Diophantine Poincar\'e estimate}
\label{sec:dioph-poincare}

The energy identity for \eqref{eq:LC} controls only the derivatives along the
fixed plane $P$.
The following estimate recovers the ambient Sobolev control from $D_Pf$.

\begin{Lemma}[Diophantine Poincar\'e estimate]
\label{lem:poincare}
If $f\in C^\infty_0(\T^N;Q)$, then for every $s\in\R$, we have
\begin{equation}
  \norm[H^s]{f}
  \le C\gamma^{-1}\norm[H^{s+\tau}]{D_Pf}.
\label{eq:poincare}
\end{equation}
\end{Lemma}

\begin{proof}
Write the Fourier expansion of $f$ as 
\[f(z)=\sum_{\xi\in \Z^N}\hat{f}(\xi)e^{2\pi i \xi\cdot z}.\]
Since $f\in C^\infty_0$, it follows that $\widehat f(0)=0$. Direct computation yields:
\[
  \sum_{j=1}^m|\widehat{D_jf}(\xi)|^2
  =|2\pi\xi_P|^2|\widehat f(\xi)|^2.
\]
Meanwhile, from the Diophantine condition we have:
\begin{align*}
  \norm[H^s]{f}^2
  &\le\gamma^{-2}
  \sum_{\xi\ne0}\langle\xi\rangle^{2(s+\tau)}
  |2\pi\xi_P|^2|\widehat f(\xi)|^2\\
  &=\gamma^{-2}\norm[H^{s+\tau}]{D_Pf}^2.
\end{align*}
Taking square roots proves the estimate.
\end{proof}

\subsection{Uniform inversion near a constant coefficient}
\label{sec:inverse}

Let
\[
  \mathcal E_{\mathrm{sym}}
  =\left\{C=(C^{ij}):
  C^{ij}\in C^\infty(\T^N;\operatorname{End}(Q)),\ 
  (C^{ij})^T=C^{ji}\right\}.
\]
Fix an integer
\begin{equation}
  s_*>\frac N2+2.
\label{eq:sstar}
\end{equation}
For $C\in\mathcal E_{\mathrm{sym}}$, write $E=C-C_0$. Sobolev embedding
shows that, for $\varepsilon_*>0$ sufficiently small,
\begin{equation}
  \norm[H^{s_*+1}]{E}<\varepsilon_*
\label{eq:Cassump-small}
\end{equation}
implies
\begin{equation}
  \sum_{i,j=1}^m\zeta_i^TC^{ij}(z)\zeta_j
  \ge\lambda\sum_{i=1}^m|\zeta_i|^2
\label{eq:Ccoercive-inverse}
\end{equation}
for all $z$ and all $(\zeta_1,\ldots,\zeta_m)\in Q^m$. Denote this
coefficient neighborhood by
\begin{equation}
  \mathfrak C_{\varepsilon_*}
  =\left\{C\in\mathcal E_{\mathrm{sym}}:
  \norm[H^{s_*+1}]{C-C_0}<\varepsilon_*\right\}.
\label{eq:coefficient-neighborhood}
\end{equation}
The coercivity above gives the pointwise quadratic bound used in the
energy identity.

The high-order estimate relies on the following standard commutator bound.

\begin{Lemma}
\label{lem:commutator}
For every integer $r\ge1$, every smooth matrix field $E$, and every smooth
vector field $v$, we have
\begin{equation}
  \left(\sum_{|\alpha|\le r}
  \norm[L^2]{[\partial^\alpha,E]v}^2\right)^{1/2}
  \le C_r\Bigl(
  \norm[L^\infty]{\nabla E}\norm[H^{r-1}]{v}
  +\norm[H^r]{E}\norm[L^\infty]{v}
  \Bigr),
\label{eq:commutator}
\end{equation}
where the commutator is defined as 
\[[\partial^\alpha, E]v:=\partial^\alpha(Ev)-E(\partial^\alpha v).\]
\end{Lemma}

A proof is given in \cref{sec:commutator-proof}. See also the related
commutator estimates in \cite{KatoPonce1988} and \cite[\S3.6]{Taylor1991}.

The next proposition establishes the uniform tame invertibility of $\mathcal{L}_C$ for coefficients $C$ sufficiently close to the fixed constant coefficient $C_0$.
\begin{Prop}
\label{prop:tame-inverse}
There exists $\varepsilon_*>0$ such that for every
$C\in\mathfrak C_{\varepsilon_*}$ and every
$f\in C^\infty_0(\T^N;Q)$, the equation
\begin{equation}
  \mathcal{L}_Cw=f
\label{eq:LCequation}
\end{equation}
has a unique solution $w\in C^\infty_0(\T^N;Q)$.  Moreover, for every
integer $s\ge0$,
\begin{equation}
  \norm[H^s]{w}
  \le A_s\left[
  \norm[H^{s+2\tau}]{f}
  +\bigl(1+\norm[H^{s+\tau}]{E}\bigr)
   \norm[H^{s_*+\tau}]{f}
  \right].
\label{eq:w-tame}
\end{equation}
The constants are uniform on the low-norm neighborhood
\eqref{eq:coefficient-neighborhood}.
\end{Prop}

\begin{proof}
By the definition of $\mathfrak C_{\varepsilon_*}$, after choosing
$\varepsilon_*>0$ sufficiently small, we have the uniform coercivity estimate for some fixed $\lambda>0$:
\begin{equation}
  \sum_{i,j=1}^m
  \zeta_i^T C^{ij}(z)\zeta_j
  \ge
  \lambda\sum_{i=1}^m|\zeta_i|^2
\label{eq:C-coercive-proof}
\end{equation}
for every $\zeta=(\zeta_1,\ldots,\zeta_m)\in P^*\otimes Q$ and every
$z\in\T^N$.

We first regularize the equation in the ambient directions.  For
$\epsilon>0$, let
\[
  \mathcal{L}_{C,\epsilon}=\mathcal{L}_C+\epsilon(-\Delta),
\]
where $-\Delta$ is the usual positive Laplacian on $\T^N$.  The
corresponding bilinear form is coercive on the zero-average $H^1$ space, so
Lax-Milgram gives a unique weak solution
\[
  \mathcal{L}_{C,\epsilon}w_\epsilon=f,
  \qquad
  \int_{\T^N}w_\epsilon=0.
\]
Since $\epsilon>0$, this is a strongly elliptic system on the ambient torus. Therefore, the standard elliptic regularity gives $w_\epsilon\in C^\infty$.

We now derive estimates independent of $\epsilon$. Throughout the proof, constants denoted by $C$ and $C_r$ may change from
line to line. They are uniform for coefficients in the fixed neighborhood
$\mathfrak C_{\varepsilon_*}$ and independent of $\epsilon,f,w_\epsilon$;
$C_r$ may also depend on $r$. Let $\partial^\alpha$
denote ambient coordinate derivatives. Because the constant vector fields
$D_i$ commute with $\partial^\alpha$, differentiating the regularized
equation gives
\begin{equation}\label{X}
  \mathcal{L}_{C,\epsilon}\partial^\alpha w_\epsilon
  =
  \partial^\alpha f
  +\sum_{i=1}^mD_i\!\left(
    \sum_{j=1}^m[\partial^\alpha,C^{ij}]D_jw_\epsilon
  \right).
\end{equation}

Since \(C=C_0+E\) and \(C_0\) is constant, we have
\[
  [\partial^\alpha,C^{ij}]
  =
  [\partial^\alpha,E^{ij}].
\]

Taking the \(L^2\)-inner product of \eqref{X} with \(\partial^\alpha w_\epsilon\) and integrating by
parts on \(\T^N\), we obtain
\begin{align*}
  &\int_{\T^N}
  \sum_{i,j=1}^m
  \left\langle
    C^{ij}D_j\partial^\alpha w_\epsilon,D_i\partial^\alpha w_\epsilon
  \right\rangle_0\,dz
  +
  \epsilon
  \norm[L^2]{\nabla \partial^\alpha w_\epsilon}^2
  \\
  &\qquad=
  \ip{\partial^\alpha f}{\partial^\alpha w_\epsilon}_{L^2}
  -
  \sum_{i=1}^m
  \ip{\sum_{j=1}^m
  [\partial^\alpha,E^{ij}]D_jw_\epsilon}{D_i\partial^\alpha w_\epsilon}_{L^2}.
\end{align*}
There are no boundary terms because all fields are periodic. By the coercivity \eqref{eq:C-coercive-proof},
\[
  \int_{\T^N}
  \sum_{i,j=1}^m
  \left\langle
    C^{ij}D_j\partial^\alpha w_\epsilon,D_i\partial^\alpha w_\epsilon
  \right\rangle_0\,dz
  \ge
  \lambda
  \sum_{i=1}^m
  \norm[L^2]{D_i\partial^\alpha w_\epsilon}^2.
\]
Summing over \(|\alpha|\le r\), retaining the nonnegative
\(\epsilon\)-term, and applying Cauchy--Schwarz to the commutator
contribution gives
\begin{align*}
  &\lambda
  \sum_{|\alpha|\le r}
  \sum_{i=1}^m
  \norm[L^2]{D_i\partial^\alpha w_\epsilon}^2
  +
  \epsilon
  \sum_{|\alpha|\le r}
  \norm[L^2]{\nabla\partial^\alpha w_\epsilon}^2
  \\
  &\qquad\le
  \sum_{|\alpha|\le r}
  \left|
    \ip{\partial^\alpha f}
       {\partial^\alpha w_\epsilon}_{L^2}
  \right|
  \\
  &\qquad\quad+
  \left(
    \sum_{|\alpha|\le r}
    \sum_{i=1}^m
    \norm[L^2]{\sum_{j=1}^m
  [\partial^\alpha,E^{ij}]D_jw_\epsilon}^2
  \right)^{1/2}
  \left(
    \sum_{|\alpha|\le r}
    \sum_{i=1}^m
    \norm[L^2]{D_i\partial^\alpha w_\epsilon}^2
  \right)^{1/2}.
\end{align*}
For integer \(r\), the derivative-based Sobolev norm is equivalent to
the Fourier Sobolev norm, so
\[
  \sum_{|\alpha|\le r}
  \sum_{i=1}^m
  \norm[L^2]{D_i\partial^\alpha w_\epsilon}^2
  \simeq
  \norm[H^r]{D_Pw_\epsilon}^2.
\]
Moreover, with the block convention
\[
  \bigl(
    [\partial^\alpha,E]D_Pw_\epsilon
  \bigr)_i
  :=
  \sum_{j=1}^m
  [\partial^\alpha,E^{ij}]D_jw_\epsilon,
\]
we have
\[
  \sum_{i=1}^m
  \norm[L^2]{\sum_{j=1}^m
  [\partial^\alpha,E^{ij}]D_jw_\epsilon}^2
  =
  \norm[L^2]{
    [\partial^\alpha,E]D_Pw_\epsilon
  }^2.
\]
Dropping the nonnegative \(\epsilon\)-term therefore yields
\begin{equation}
\begin{aligned}
  \lambda \norm[H^r]{D_Pw_\epsilon}^2
  \le{}&
  C_r\sum_{|\alpha|\le r}
  \left|
    \ip{\partial^\alpha f}{\partial^\alpha w_\epsilon}_{L^2}
  \right| \\
  &+C_r
  \left(
    \sum_{|\alpha|\le r}
    \norm[L^2]{[\partial^\alpha,E]D_Pw_\epsilon}^2
  \right)^{1/2}
  \norm[H^r]{D_Pw_\epsilon},
\end{aligned}
\label{eq:energy-r-continuous}
\end{equation}
for a constant $C_r>0$ independent of $\epsilon$.

For the forcing term, we use Sobolev duality followed by the Diophantine
Poincar\'e inequality of \cref{lem:poincare}:
\begin{equation}\label{Diophantine1}
  \sum_{|\alpha|\le r}
  \left|
    \ip{\partial^\alpha f}{\partial^\alpha w_\epsilon}_{L^2}
  \right|
  \le
  C_r
  \norm[H^{r+\tau}]{f}
  \norm[H^{r-\tau}]{w_\epsilon}
  \le
  C_r
  \norm[H^{r+\tau}]{f}
  \norm[H^r]{D_Pw_\epsilon}.
\end{equation}
 Since $s_*>N/2+2$, Sobolev embedding \eqref{eq:sobolev-embedding} yields
\[
  \norm[L^\infty]{\nabla E}
  \le C\norm[H^{s_*+1}]{E},
  \qquad
  \norm[L^\infty]{D_Pw_\epsilon}
  \le C\norm[H^{s_*}]{D_Pw_\epsilon}.
\]
Combining these with the tame commutator estimate \eqref{eq:commutator} with $v=D_Pw_\epsilon$, we get
\begin{equation}
\begin{aligned}
  &\left(
    \sum_{|\alpha|\le r}
    \norm[L^2]{[\partial^\alpha,E]D_Pw_\epsilon}^2
  \right)^{1/2} \\
  &\qquad\le
  C_r\left(
    \norm[H^{s_*+1}]{E}\norm[H^{r-1}]{D_Pw_\epsilon}
    +\norm[H^r]{E}\norm[H^{s_*}]{D_Pw_\epsilon}
  \right).
\end{aligned}
\label{eq:commutator-tame-continuous}
\end{equation}
If $\norm[H^r]{D_Pw_\epsilon}=0$, the next estimate is immediate.
Otherwise, combining \eqref{eq:energy-r-continuous}, \eqref{Diophantine1}
and \eqref{eq:commutator-tame-continuous} and dividing by
$\norm[H^r]{D_Pw_\epsilon}$ gives
\begin{equation}\label{X2}
  \begin{aligned}
  \norm[H^r]{D_Pw_\epsilon}
  \le{}&
  C_r\norm[H^{r+\tau}]{f} \\
  &+C'_r
  \left(\norm[H^{s_*+1}]{E}\norm[H^{r-1}]{D_Pw_\epsilon}
    +\norm[H^r]{E}\norm[H^{s_*}]{D_Pw_\epsilon}
    \right)
\end{aligned}
\end{equation}
Taking first $r=s_*$ and using
\[ \norm[H^{s_*-1}]{D_Pw_\epsilon}\leq \norm[H^{s_*}]{D_Pw_\epsilon} \]
and
\[
  \norm[H^{s_*}]{E}\le \norm[H^{s_*+1}]{E}<\varepsilon_*,
\]
we may shrink $\varepsilon_*$ so that the term containing
$\norm[H^{s_*+1}]{E}\norm[H^{s_*}]{D_Pw_\epsilon}$ is absorbed into the
left-hand side.  We obtain the uniform baseline estimate:
\begin{equation}
  \norm[H^{s_*}]{D_Pw_\epsilon}
  \le
  C\,\norm[H^{s_*+\tau}]{f}.
\label{eq:baseline-DPw}
\end{equation}
For arbitrary $r> s_*$, we cannot simply choose one neighborhood $\varepsilon_*$ such that $C_r'\varepsilon_*$ is uniformly small for all $r$, since $C_r'$ can grow with $r$. Instead, we use \cref{lem:sobolev-interpolation}. When $r=s_*+1$, the same estimate is immediate because $r-1=s_*$. For all $\delta>0$ there exists $C_{r,\delta}>0$ such that
\begin{equation}
  \norm[H^{r-1}]{D_Pw_\epsilon}\leq \delta \norm[H^r]{D_Pw_\epsilon}+C_{r,\delta}\norm[H^{s_*}]{D_Pw_\epsilon}.
\end{equation}
Substituting this interpolation estimate into $\eqref{X2}$ gives
\begin{equation}\label{X3}
  \begin{aligned}
  \norm[H^r]{D_Pw_\epsilon}
  \le{}&
  C_r\norm[H^{r+\tau}]{f} \\
  &+C'_r
  \left(\varepsilon_*\left( \delta \norm[H^r]{D_Pw_\epsilon}+C_{r,\delta}\norm[H^{s_*}]{D_Pw_\epsilon}\right)
    +\norm[H^r]{E}\norm[H^{s_*}]{D_Pw_\epsilon}
    \right)
\end{aligned}
\end{equation}
We can then choose $\delta=\delta_r$ small such that $C_r'\varepsilon_*\delta\norm[H^r]{D_Pw_\epsilon}$ can be absorbed into the left side. It follows that
\begin{equation}
  \norm[H^r]{D_Pw_\epsilon}
  \le
  C_r\left(
    \norm[H^{r+\tau}]{f}
    +(1+\norm[H^r]{E})\norm[H^{s_*+\tau}]{f}
  \right),
  \qquad r\ge s_*.
\label{eq:DPw-tame}
\end{equation}

The estimate \eqref{eq:DPw-tame} controls the leafwise derivatives of
$w_\epsilon$.  To recover the full ambient Sobolev norm, we use the
Diophantine Poincar\'e inequality once more:
\[
  \norm[H^s]{w_\epsilon}
  \le
  C\gamma^{-1}
  \norm[H^{s+\tau}]{D_Pw_\epsilon}.
\]
If $s+\tau\ge s_*$, apply \eqref{eq:DPw-tame} with $r=s+\tau$. If
$s+\tau<s_*$, we use \eqref{eq:baseline-DPw} and the
monotonicity of Sobolev norms.  In both cases,
\begin{equation}
  \norm[H^s]{w_\epsilon}
  \le
  C_s\left[
    \norm[H^{s+2\tau}]{f}
    +\bigl(1+\norm[H^{s+\tau}]{E}\bigr)
      \norm[H^{s_*+\tau}]{f}
  \right].
\label{eq:w-epsilon-tame}
\end{equation}

The bounds above are uniform as $\epsilon\downarrow0$. Therefore, a diagonal weak
compactness argument gives a sequence $\epsilon_\nu\downarrow0$
and a function $w\in C^\infty_0(\T^N;Q)$ such that
$w_{\epsilon_\nu}\rightharpoonup w$ in every Sobolev space.  For every smooth test field $\varphi$, the regularizing term satisfies
\[
  \epsilon_\nu\int_{\T^N}
  \langle\nabla w_{\epsilon_\nu},\nabla\varphi\rangle\,dz\longrightarrow0,
\]
because $w_{\epsilon_\nu}$ is uniformly bounded in $H^1$. Passing to the
weak formulation gives
\[
  \mathcal{L}_Cw=f.
\]
Weak lower semicontinuity shows that $w$ satisfies the same tame estimate.

It remains only to prove uniqueness.  If $\mathcal{L}_Cw=0$ and
$\int_{\T^N}w=0$, then by block symmetry and coercivity,
\[
  0
  =
  \int_{\T^N} w^T\mathcal{L}_Cw
  =
  \int_{\T^N}
  \sum_{i,j=1}^m(D_iw)^TC^{ij}D_jw
  \ge
  c\norm[L^2]{D_Pw}^2.
\]
Hence $D_Pw=0$.  By \cref{lem:poincare} and the zero-average condition,
$w=0$.  Thus $\mathcal{L}_C$ is bijective on $C^\infty_0(\T^N;Q)$, and the
preceding estimate proves that, for fixed $C\in\mathfrak C_{\varepsilon_*}$,
\[
  f\longmapsto \mathcal{L}_C^{-1}f
\]
is a tame linear map. The proof is complete.
\end{proof}

\subsection{The solution map is smooth tame}
\label{sec:smooth-coefficient-dependence}

\begin{Prop}
\label{prop:smooth-inverse}
The solution map
\[
  S:\mathfrak C_{\varepsilon_*}\times
  C^\infty_0(\T^N;Q)
  \longrightarrow C^\infty_0(\T^N;Q),
  \qquad
  S(C,f)=\mathcal{L}_C^{-1}f,
\]
is smooth tame.
\end{Prop}

\begin{proof}
The map
\[
  (C,w)\longmapsto \mathcal{L}_Cw
  =-\sum_{i=1}^mD_i\!\left(
    \sum_{j=1}^mC^{ij}D_jw
  \right)
\]
is smooth tame and linear in $w$ by the tame product estimate
\eqref{eq:tame-product}.  \Cref{prop:tame-inverse} gives
unique solvability of $\mathcal{L}_Cw=f$ for every
$C\in\mathfrak C_{\varepsilon_*}$, and the estimate
\eqref{eq:w-tame} shows that the solution map $S(C,f)=\mathcal{L}_C^{-1}f$
is tame.

It remains to prove continuity.  Let
\[
  w=S(C,f),
  \qquad
  w'=S(C',f').
\]
Subtracting the two equations gives
\[
  \mathcal{L}_C(w'-w)
  =f'-f+\sum_{i=1}^mD_i\!\left(
    \sum_{j=1}^m(C'-C)^{ij}D_jw'
  \right).
\]
As $(C',f')\to(C,f)$ in the Fr\'echet topology, the tame product
estimate and the local bounds for $w'$ supplied by \eqref{eq:w-tame}
show that the right-hand side tends to zero in every Sobolev norm.
Applying \cref{prop:tame-inverse} to $w'-w$ therefore gives
\[
  \norm[H^s]{w'-w}\longrightarrow0
  \qquad\text{for every }s.
\]
Thus $S$ is continuous and tame.  Hamilton's smooth tame
inverse-family theorem, \cref{thm:hamilton-inverse-family}, implies
that $S$ is smooth tame.
\end{proof}

\section{Construction of the approximate right inverse and the quadratic error}
\label{sec:approx}
We now return to the original geometric setting.  \cref{sec:linearization} produced the coefficient
$C=C(u,g)$ and proved that it depends smoothly and tamely on $(u,g)$, while
\cref{sec:inverse-main} established smooth tame inversion for the corresponding abstract
operator $\mathcal{L}_C$.  We combine these two ingredients to construct the maps $V$
and $\mathscr Q$ required by Hamilton's theorem. The derivative of the fixed-range map is
\begin{equation}
  D_u\cP(u,g)[v]=\Pi_0D_u\cH(u,g)[v].
\label{eq:DP}
\end{equation}

Fix $0<\varepsilon_0<\varepsilon_*$.  By
\cref{lem:coefficient-tame} and Sobolev embedding, we may choose an integer
$r_0$ sufficiently large and then $\rho>0$ sufficiently small so that the set
\begin{equation}
 \mathcal O=
 \left\{(u,g)\in X\times\cM^\infty:
 \norm[H^{r_0}]{u}+\norm[H^{r_0}]{g-g_0}<\rho\right\}
\label{eq:admissible-neighborhood}
\end{equation}
has the following properties: $\Phi_u$ is a diffeomorphism, $J=I_Q+\partial_Qu$
is pointwise invertible, $g$ is uniformly positive definite, and
\begin{equation}
 \norm[H^{s_*+1}]{C(u,g)-C_0}\le \varepsilon_0.
\label{eq:section7-coefficient-bound}
\end{equation}
All constructions in this section take place on $\mathcal O$.

For $h\in Y$, define
\begin{equation}
\begin{gathered}
  b_h=h-\int_{\T^N}J^Th\,dz,\\
  w_h=S\bigl(C(u,g),J^Tb_h\bigr),
  \qquad
  \psi_h=w_h-\int_{\T^N}Jw_h\,dz.
\end{gathered}
\label{eq:psi-explicit}
\end{equation}
Because $\int J\,dz=I_Q$ and $\int h\,dz=0$, we have
\begin{equation}
  \Pi_0b_h=h,
  \qquad
  \int_{\T^N}J^Tb_h\,dz=0.
\label{eq:bh-properties}
\end{equation}
Thus $w_h$ is well defined and
\begin{equation}
  \mathcal{L}_{C(u,g)}\psi_h=J^Tb_h,
  \qquad
  \int_{\T^N}J\psi_h\,dz=0.
\label{eq:psi-equations}
\end{equation}
Define
\begin{equation}
  V(u,g)h:=J\psi_h.
\label{eq:V}
\end{equation}
Then $V(u,g)h\in X$. The formulas above, together with
\cref{prop:smooth-inverse}, show directly that $(u,g,h)\mapsto V(u,g)h$ is
smooth tame and linear in $h$.

\begin{Prop}
\label{prop:approx-inverse}
On the open neighborhood $\mathcal O$ defined in
\eqref{eq:admissible-neighborhood}, there exist smooth tame maps
\[
 V:\mathcal O\times Y\longrightarrow X,
 \qquad
 \mathscr Q:\mathcal O\times Y\times Y\longrightarrow Y,
\]
linear in the last variable of $V$ and bilinear in the last two variables of
$\mathscr Q$, such that
\begin{equation}
 D_u\cP(u,g)V(u,g)h
 =h+\mathscr Q(u,g)\{\cP(u,g),h\}.
\label{eq:quadratic-identity}
\end{equation}
In particular, whenever $\cP(u,g)=0$, $V(u,g)$ is an exact right inverse
of $D_u\cP(u,g)$.
\end{Prop}

\begin{proof}
Let $h\in Y$ and define $\psi_h,b_h$ by \eqref{eq:psi-explicit}.  We apply the ground-state identity
\eqref{eq:ground-state} to $\psi_h$.  After absorbing the divergence-free
part of $K$ into the coefficient $C$, equations \eqref{eq:psi-equations}
and \eqref{eq:Cdef} give
\begin{equation}
 J^T\Bigl(
 D_u\cH(u,g)[J\psi_h]
 -\sum_{b=1}^q(\partial_b\cH(u,g))\psi_h^b
 \Bigr)
 =J^Tb_h-\sum_{j=1}^mK_{\parallel}^jD_j\psi_h.
\label{eq:ground-applied}
\end{equation}
Multiplication by $(J^{T})^{-1}$ yields
\begin{equation}
 D_u\cH(u,g)[V(u,g)h]
 =b_h+\sum_{b=1}^q(\partial_b\cH(u,g))\psi_h^b
 -(J^{T})^{-1}\sum_{j=1}^mK_{\parallel}^jD_j\psi_h.
\label{eq:DE-V}
\end{equation}
By \cref{lem:fixed-target}, $\cH(u,g)-\cP(u,g)$ is spatially constant, and hence
\[
 \partial_b\cH(u,g)=\partial_b\cP(u,g).
\]
By \cref{lem:Kparallel},
\[
 K_{\parallel}=T_u(\cP(u,g)).
\]
Apply $\Pi_0$ to \eqref{eq:DE-V}, use \eqref{eq:DP} and
$\Pi_0b_h=h$, and define
\begin{equation}
 \mathscr Q(u,g)\{r,h\}
 =\Pi_0\left[
 \sum_{b=1}^q(\partial_br)\psi_h^b
 -(J^{T})^{-1}\sum_{j=1}^m(T_u(r))^jD_j\psi_h
 \right].
\label{eq:Qformula}
\end{equation}
This proves \eqref{eq:quadratic-identity}.  The map
$h\mapsto\psi_h$ is linear, and $r\mapsto T_u(r)$ is linear, so
\eqref{eq:Qformula} is bilinear in $(r,h)$.

The map $(u,g,h)\mapsto\psi_h$ is smooth tame and linear in $h$.
Matrix inversion
$J\mapsto (J^{T})^{-1}$, the residual operator $(u,r)\mapsto T_u(r)$,
differentiation, multiplication, and $\Pi_0$ are smooth tame operations.
Composing these maps in \eqref{eq:Qformula} proves that $\mathscr Q$ is smooth
tame.
\end{proof}

At an exact solution, the linearization is an isomorphism. 
\begin{Cor}
\label{cor:linearized-isomorphism}
Let $(u,g)\in\mathcal O$ satisfy
\[
  \cP(u,g)=0.
\]
Then
\[
  D_u\cP(u,g):X\longrightarrow Y
\]
is an isomorphism, and its inverse is $V(u,g)$.
\end{Cor}

\begin{proof}
By \cref{prop:approx-inverse},
\[
  D_u\cP(u,g)V(u,g)=I_Y,
\]
so $D_u\cP(u,g)$ is surjective. It remains to prove injectivity.

Let $v\in X$ satisfy $D_u\cP(u,g)[v]=0$.
Since
\[
  D_u\cP(u,g)[v]
  =\Pi_0D_u\cH(u,g)[v],
\]
there is a constant vector $c\in Q$ such that
\[
  D_u\cH(u,g)[v]=c.
\]
Because $\cP(u,g)=0$, \cref{lem:fixed-target} gives
\[
  \cH(u,g)=0.
\]
Differentiating the identity \eqref{eq:noether} in the direction
$v$, the term obtained by differentiating $J$ vanishes, hence
\[
  0
  =\int_{\T^N}J^TD_u\cH(u,g)[v] \,dz
  =\left(\int_{\T^N}J^T\,dz\right)c
  =c.
\]
Here we have used \eqref{eq:meanJ}. Therefore
\[
  D_u\cH(u,g)[v]=0.
\]

Set
\[
  w=J^{-1}v, \quad w_0=\Pi_0 w.
\]
Since $\cP(u,g)=\cH(u,g)=0$, the reduced identity
\eqref{eq:reduced-ground-state} gives
\[
  \mathcal{L}_{C(u,g)}w=J^TD_u\cH(u,g)[v]=0.
\]
The operator $\mathcal{L}_{C(u,g)}$ annihilates constants, so
\[
  \mathcal{L}_{C(u,g)}w_0=0.
\]
Because $(u,g)\in\mathcal O$, the coefficient $C(u,g)$ belongs to
$\mathfrak C_{\varepsilon_*}$. Therefore, the uniqueness statement in
\cref{prop:tame-inverse} yields
\[
  w_0=0.
\]
Thus $w=a$ is constant and
\[
  v=Ja.
\]
Since $v\in X$, its average vanishes, and \eqref{eq:meanJ} gives
\[
  0
  =\int_{\T^N}v\,dz
  =\left(\int_{\T^N}J\,dz\right)a
  =a.
\]
Hence $v=0$, proving injectivity. Since $V(u,g)$ is already a right
inverse, it is the inverse of $D_u\cP(u,g)$.
\end{proof}

\section{Proofs of the main theorems}
\label{sec:nashmoser}

\subsection{Proof of \cref{thm:main}}
\label{sec:proof-main}
We now apply \cref{thm:hamilton} to the Euler-Lagrange map constructed
earlier.

\begin{proof}
We recall here the tame Fr\'echet spaces
\[
  F=X=C^\infty_0(\T^N;Q),
  \qquad
  H=Y=C^\infty_0(\T^N;Q),
\]
and take the parameter space $G$ to be the space of smooth symmetric
two-tensors on $\T^N$, with $g$ belonging to an open subset
$\cM^\infty$ near $g_0$. Define
\begin{equation}
  A(u,g)=\cP(u,g)=\Pi_0\cH(u,g).
\label{eq:A-Hamilton}
\end{equation}
By \cref{lem:H-smooth-tame}, the Euler-Lagrange map $\cH$ is smooth tame. Since the projection $\Pi_0$ is tame linear, it follows that $A$ is smooth tame.

At the base point, the affine Euclidean leaves are minimal, so
\[
  A(0,g_0)=\cP(0,g_0)=0.
\]
Choose $s_*>N/2+2$ and shrink an open neighborhood of $(0,g_0)$ in one fixed
high Sobolev norm if necessary so that:
\begin{enumerate}[label=\textup{(\roman*)}]
\item $\Phi_u$ is a diffeomorphism and $J$ is pointwise invertible;
\item the metric remains positive definite;
\item \cref{lem:coefficient-tame} places $C(u,g)$ in the coefficient
  neighborhood of \cref{prop:tame-inverse}.
\end{enumerate}
This neighborhood is open in the $C^\infty$ Fr\'echet topology because a fixed
Sobolev seminorm is continuous.

\Cref{prop:smooth-inverse,prop:approx-inverse} provide
smooth tame maps $V$ and $\mathscr Q$ satisfying exactly
\eqref{eq:Hamilton-hyp}.  Therefore, \cref{thm:hamilton} produces a smooth tame
map $g\mapsto u(g)\in X$, with $u(g_0)=0$, such that
\begin{equation}
  \cP(u(g),g)=0.
\label{eq:Pzero-final}
\end{equation}
It follows from \cref{lem:fixed-target} that \eqref{eq:Pzero-final} is equivalent to
$\cH(u(g),g)=0$.  Shrinking the metric neighborhood once more ensures that
$\norm[C^1]{u(g)}$ is small, so $\Phi_{u(g)}$ is an equivariant diffeomorphism.
By \cref{prop:H-mean-curvature}, every leaf
$\Phi_{u(g)}(P+y)$ is an embedded minimal submanifold of $(\R^N,g)$.
Equivariance under $\Z^N$ makes the foliation $\Z^N$-invariant, so it
descends to the quotient torus. Since the affine leaves are dense by
\cref{rem:dense-leaves} and the descended map is a diffeomorphism, every
quotient leaf is dense.

\medskip
\noindent\emph{Local uniqueness.}
The smooth tame family $V(u,g)y$ is linear in $y$, while
$D_u^2\cP(u,g)[h_1,h_2]$ is smooth tame and bilinear in its last
two arguments. By the homogeneous tame estimates
\eqref{eq:uniform-tame-form}, applied with $(u,g-g_0)$ as parameter,
there exist fixed integers $\ell,b\ge 1$ and a constant $C>0$ such that,
after increasing $r_0$ and shrinking $\rho$ in
\eqref{eq:admissible-neighborhood} if necessary, we have:
\[
  \norm[L^2]{V(u,g)y}\le C\norm[H^\ell]{y},
  \qquad
  \norm[H^\ell]{D_u^2\cP(u',g)[h,h]}
  \le C\norm[H^b]{h}^2,
\]
uniformly for all $(u,g),(u',g)\in\mathcal O$, $y\in Y$,
and $h\in X$. Here $r_0$ is chosen large enough to bound
all the fixed parameter norms appearing in these estimates.
Shrink $\cG$ so that, for every $g\in\cG$ we have
\[
  \norm[H^{r_0}]{u(g)}+\norm[H^{r_0}]{g-g_0}<\frac{\rho}{2},
\]
and set $K_\mathcal{G}=\max\{2b,r_0\}$.
Fix $g\in\cG$, write $u=u(g)$, and let $\widetilde u\in X$ determine
another minimal graph foliation for the same metric. Put
$h=\widetilde u-u$. Choose $0<\eta\le\rho/2$, to be decreased below,
and suppose that $\norm[H^{K_\mathcal{G}}]{h}<\eta$.
Since $K_\mathcal{G}\ge r_0$, and $\eta<\rho/2$, it follows that
$(u+th,g)\in\mathcal O$ for every $0\le t\le1$.
By \cref{prop:H-mean-curvature,lem:fixed-target},
$\cP(u+h,g)=\cP(u,g)=0$, so Taylor's formula gives
\[
  0=D_u\cP(u,g)[h]
    +\int_0^1(1-t)D_u^2\cP(u+th,g)[h,h]\,dt.
\]
At the exact solution, \cref{cor:linearized-isomorphism} gives the
left-inverse identity $V(u,g)D_u\cP(u,g)=I_X$. Hence
\[
  h=-V(u,g)\int_0^1(1-t)D_u^2\cP(u+th,g)[h,h]\,dt,
\]
and the fixed-order estimates imply
\begin{equation}
\begin{aligned}
  \norm[L^2]{h}
  &\le C\int_0^1(1-t)
      \norm[H^\ell]{D_u^2\cP(u+th,g)[h,h]}\,dt\\
  &\le C\norm[H^b]{h}^2.
\end{aligned}
\label{eq:uniqueness-basic}
\end{equation}

As in Moser's uniqueness argument \cite[\S5]{Moser1988}, applying
\cref{lem:sobolev-interpolation} between orders $0$ and $2b$, we have
\[
  \norm[H^b]{h}^2
  \le \norm[L^2]{h}\norm[H^{2b}]{h}
  \le \norm[L^2]{h}\norm[H^{K_\mathcal{G}}]{h}.
\]
Combining this with \eqref{eq:uniqueness-basic} gives
\[
  \norm[L^2]{h}
  \le C\norm[L^2]{h}\norm[H^{K_\mathcal{G}}]{h}
  \le C\eta\norm[L^2]{h}.
\]
Choose $\eta$ smaller if necessary so that $C\eta<1$.
Then $h=0$, and consequently $\widetilde u=u(g)$.
The choices of $K_\mathcal{G}$ and $\eta$ are uniform for $g\in\cG$.

The final metric neighborhood $\cG$ is open in the Fr\'echet topology.
Therefore, it contains an open neighborhood determined by finitely many
Sobolev seminorms. Since the Sobolev grading is increasing, these
conditions are implied by $\norm[H^{r_0}]{g-g_0}<\varepsilon$ for some
$\varepsilon>0$, after increasing the index $r_0$ in the statement of
\cref{thm:main} if necessary. This completes the proof.
\end{proof}

\subsection{Proof of \cref{thm:calibrated}}
\label{sec:calibration}

\begin{proof}
We use the orthonormal bases $e_1,\ldots,e_m$ of $P$ and
$n_1,\ldots,n_q$ of $Q$ fixed in \cref{sec:geometry}, and give $P$
the orientation determined by $e_1,\ldots,e_m$.
For $g\in\cG$, write $u=u(g)$ and set
\[
  \widetilde g=\Phi_u^*g.
\]
By \cref{thm:main}, the affine foliation $\{P+y:y\in Q\}$ is
minimal for $\widetilde g$. Since $u(g_0)=0$ and $g\mapsto u(g)$ is
smooth, after shrinking $\cG$ the metric $\widetilde g$ is as close
to $g_0$ in any prescribed finite $C^r$ norm as needed below.

Let $\chi$ be the characteristic $m$-form of the affine foliation
for $\widetilde g$:
\[
  \chi=
  \frac{\widetilde g(e_1,\cdot)\wedge\cdots\wedge
        \widetilde g(e_m,\cdot)}
       {\mu_{0,\widetilde g}},
\]
where $\mu_{0,\widetilde g}$ is the area density in
\eqref{eq:area-density}.
By definition, $\chi$ has comass one and restricts to the oriented
$\widetilde g$-volume form on $P$.
It follows from \cref{areaform}, evaluated at $u=0$ that
\[
  \chi(e_1,\ldots,e_{i-1},n_a,e_{i+1},\ldots,e_m)
  =(L_{\widetilde g})_{p_i^a}(z,0,0).
\]
Since the fixed vector fields commute and
$\partial_a\mu_{0,\widetilde g}=(L_{\widetilde g})_{r^a}(z,0,0)$,
\eqref{eq:H-definition} yields
\begin{equation}
\begin{aligned}
  d\chi(n_a,e_1,\ldots,e_m)
  &=\partial_a\mu_{0,\widetilde g}
    -\sum_{i=1}^mD_i\bigl((L_{\widetilde g})_{p_i^a}(z,0,0)\bigr)\\
  &=\cH(0,\widetilde g)_a=0.
\end{aligned}
\label{eq:calibration-filtration}
\end{equation}

We apply $D_i$, $\Delta_P^{-1}$, and $\Pi_0$ to differential forms
coefficientwise. The coefficients of $d\chi$ are
periodic derivatives and have zero average, so
\cref{lem:leafwise-laplacian} allows us to define the following smooth $m$-form:
\[
  \beta=\sum_{i=1}^m\iota_{e_i}D_i\Delta_P^{-1}(d\chi),
\]
where $\iota_{e_i}$ denotes contraction with $e_i$.
On forms, the Cartan formula yields $d\iota_{e_i}+\iota_{e_i}d=D_i$. Since $d$ commutes
with $D_i$ and $\Delta_P^{-1}$ and $d^2\chi=0$, we obtain
\[
  d\beta
  =\sum_{i=1}^mD_i^2\Delta_P^{-1}(d\chi)
  =-d\chi.
\]
Hence $\chi+\beta$ is closed. From the definition of $\beta$ and \eqref{eq:calibration-filtration}, it follows that $\beta(v_1,v_2,\ldots, v_m)=0$ whenever at least $m-1$ of $v_i$ belong to $P$. In particular, $\chi$ and $\chi+\beta$ agree on $P$
with the induced volume form.

For a fixed integer $s>N/2$, it follows from Sobolev embedding and \cref{lem:Kparallel} that:
\begin{equation}
  \norm[C^0]{\beta} \le C\norm[H^s]{\beta}
  \le C\gamma^{-1}\norm[H^{s+\tau}]{d\chi}
  \le C\gamma^{-1}\norm[H^{s+\tau+1}]{\Pi_0\chi}.
\label{eq:calibration-homotopy-estimate}
\end{equation}
At $\widetilde g=g_0$, the characteristic form $\chi$ has constant
coefficients, so $\Pi_0\chi=0$. Since $\chi$ depends smoothly on
$\widetilde g$, we have
\[
\|\Pi_0\chi\|_{H^{s+\tau+1}}\longrightarrow 0
\qquad\text{as }\widetilde g\longrightarrow g_0.
\]

We now verify the comass bound of $\chi+\beta$. At any point, let $\zeta=w_1\wedge w_2\wedge\cdots\wedge w_m$ be an oriented
$\widetilde g$-unit simple $m$-vector. Choose $\widetilde g$-orthonormal basis $E_1, \ldots, E_m$ for $P$ and unit normal vector $N_i$ such that
\[w_i=\cos \theta_i E_i+\sin \theta_i N_i, \quad 0\le \theta_i\le \dfrac{\pi}{2}.\]
We have $|\chi(\zeta)|=\prod \cos\theta_i$ and 
\[\beta(\zeta)=\beta(\cos\theta_1 E_1+\sin\theta_1 N_1, \ldots, \cos\theta_m E_m+\sin\theta_m N_m).\]
Expanding this out and note that the terms that survive must contain at least two distinct $N_i, N_j$, and they are bounded by \[\norm[C^0]{\beta}\sin \theta_i\sin \theta_j \leq \norm[C^0]{\beta}(1-\cos\theta_i\cos \theta_j)\leq \norm[C^0]{\beta}(1-|\chi(\zeta)|).\]
Since there are at most $2^m$ terms in the expansion, we have
\[
  |\beta(\zeta)|
  \le 2^m|\beta|_{\widetilde g}
       \bigl(1-|\chi(\zeta)|\bigr).
\]

By $\eqref{eq:calibration-homotopy-estimate}$, 
we can shrink $\cG$ if necessary so that
$2^m\norm[C^0(\widetilde g)]{\beta}\le\tfrac12$.
It follows that
\[
  |(\chi+\beta)(\zeta)|
  \le |\chi(\zeta)|+\tfrac12(1-|\chi(\zeta)|)
  \le1.
\]
Thus $\chi+\beta$ is a smooth closed calibration of every
affine leaf for $\widetilde g$.

Finally, define
\[
  \omega_g=(\Phi_u^{-1})^*(\chi+\beta).
\]
The diffeomorphism $\Phi_u$ is an isometry from
$(\T^N,\widetilde g)$ to $(\T^N,g)$, so $\omega_g$ is closed,
has $g$-comass at most one, and restricts to the induced volume
form on each $\Sigma_y(g)$ with its induced orientation.
Its lift is a calibration on $\R^N$ with respect to $ g$ that calibrates every leaf $\Sigma_y(g)$ of the perturbed foliation. It follows from standard argument in \cite{HarveyLawson1982Calibrated} that the leaves are area-minimizing. This completes the proof of \cref{thm:calibrated}.

\end{proof}

\section{A variational extension to elliptic integrands}
\label{sec:general-integrands}

We briefly discuss to what extent the stability results cover a wider class of periodic graph integrands. Recall that the proof of the main theorem uses only the variational Hessian symmetry, the dependence on the
physical point $z+r$, and the coercivity of the Hessian in the slope
variable. Therefore, it extends to a controlled class of periodic graph
integrands.
Periodic elliptic integrands have also been studied from the global
minimizing viewpoint. In codimension one, Caffarelli and de la Llave
constructed plane-like minimizers and associated laminations
\cite{CaffarelliDeLaLlave2001}. The result below is instead a local smooth
persistence theorem in arbitrary codimension.

Fix $\rho>0$, let $B_\rho\subset\Hom(P,Q)$ be the ball of radius $\rho$, and
let
\[
  \mathscr F:\T^N\times\overline B_{2\rho}\longrightarrow\R
\]
be smooth up to the boundary. For
$z\in\T^N$, $r\in Q$, and $p\in B_\rho$, set
\begin{equation}
  L_{\mathscr F}(z,r,p)=\mathscr F(z+r,p).
\label{eq:general-L}
\end{equation}
For $y\in Q$ and a relatively compact domain $\Omega\Subset P+y$,
define the local graph functional
\[
  \mathcal A_{\mathscr F,\Omega}(u)
  :=
  \int_\Omega
  L_{\mathscr F}(z,u(z),D_Pu(z))\,d\mathcal H_0^m(z),
\]
where $u$ is a smooth $Q$-valued map defined near
$\overline\Omega$, with $D_Pu\in B_\rho$.
For every $h\in C_c^\infty(\Omega;Q)$, we have
\[
  \left.\frac{d}{dt}\right|_{t=0}
  \mathcal A_{\mathscr F,\Omega}(u+th)
  =
  \int_\Omega
  \langle\mathcal H_{\mathscr F}(u),h\rangle_0\,
  d\mathcal H_0^m,
\]

where the Euler-Lagrange map is
\begin{equation}
  \mathcal H_{\mathscr F}(u)_a
  =
  (L_{\mathscr F})_{r^a}(z,u,D_Pu)
  -
  \sum_{i=1}^m D_i
  \bigl((L_{\mathscr F})_{p_i^a}(z,u,D_Pu)\bigr).
  \label{eq:general-Euler}
\end{equation}

Set
\begin{equation}
  \mathcal P_{\mathscr F}(u)=\Pi_0\mathcal H_{\mathscr F}(u).
\label{eq:general-projected}
\end{equation}

Let the reference integrand $\mathscr F_0=\mathscr F_0(p)$ be independent of
the ambient point, and put
\[
  (A_0^{ij})_{ab}
  =\frac{\partial^2\mathscr F_0}
  {\partial p_i^a\partial p_j^b}(0).
\]
Assume the strong ellipticity condition
\begin{equation}
  \sum_{i,j=1}^m\zeta_i^TA_0^{ij}\zeta_j
  \ge\lambda_0\sum_{i=1}^m|\zeta_i|^2
  \qquad
  \text{for every }(\zeta_1,\ldots,\zeta_m)\in P^*\otimes Q
\label{eq:general-strong-ellipticity}
\end{equation}
for some $\lambda_0>0$. Equivalently,
$D_p^2\mathscr F_0(0)$ is positive definite on the full matrix space
$P^*\otimes Q$. This gives the pointwise coercivity used in the energy
estimate of \cref{sec:inverse}.

\begin{Cor}
\label{cor:general-integrand}
Assume that $P$ is Diophantine relative to $\Z^N$ and that
$\mathscr F_0$ satisfies \eqref{eq:general-strong-ellipticity}. After
possibly shrinking $\rho$, there is a $C^\infty$-open neighborhood
$\mathfrak I$ of $\mathscr F_0$ in the tame Fr\'echet space of functions
smooth up to the boundary of
$\T^N\times\overline B_{2\rho}$ and a smooth tame map
\[
  \mathfrak I\longrightarrow C^\infty_0(\T^N;Q),
  \qquad
  \mathscr F\longmapsto u(\mathscr F),
\]
with $u(\mathscr F_0)=0$, such that
\[
  \mathcal H_{\mathscr F}(u(\mathscr F))=0.
\]
For $\mathscr F$ sufficiently close to $\mathscr F_0$, the family
$\{\Phi_{u(\mathscr F)}(P+y):y\in Q\}$ is a $\Z^N$-periodic
foliation. For every $y\in Q$ and $\Omega\Subset P+y$,
the restriction $u(\mathscr F)|_\Omega$ is stationary for
$\mathcal A_{\mathscr F,\Omega}$ under compactly supported
$Q$-valued variations.
\end{Cor}

\begin{proof}
The dependence on $z+r$ gives the same relabelling invariance as in
\cref{sec:gauge}. Hence the identity \eqref{eq:noether}, the projected map
\eqref{eq:general-projected}, and the fixed-range identity remain valid.
Moreover,
\[
  (L_{\mathscr F})_{z_a}=(L_{\mathscr F})_{r^a},
\]
so the matrix ground-state identity, the normal-curl formula, and the
decomposition in \cref{sec:linearization} carry over without change. Therefore, the
transformed equation is
\[
\begin{aligned}
  J^T\Bigl(
  D_u\mathcal H_{\mathscr F}(u)[Jw]
  -\sum_{a=1}^q(\partial_a\mathcal H_{\mathscr F}(u))w^a
  \Bigr)
  ={}&-\sum_{i=1}^mD_i\!\left(
       \sum_{j=1}^mC_{\mathscr F}^{ij}(u)D_jw
     \right)\\
  &-\sum_{j=1}^m
    \bigl(T_u(\mathcal P_{\mathscr F}(u))\bigr)^jD_jw.
\end{aligned}
\]
At the reference state,
$C_{\mathscr F_0}^{ij}(0)=A_0^{ij}$. The Hessian symmetry gives block
symmetry, and \eqref{eq:general-strong-ellipticity} gives the constant
coercivity required in \eqref{eq:C0-coercive}. Thus
\cref{prop:tame-inverse,prop:smooth-inverse} apply with
$C_0^{ij}=A_0^{ij}$. By \cref{prop:composition-evaluation}\textup{(ii)}, evaluation of
$\mathscr F$ and its finite derivatives along $(z+u,D_Pu)$ is jointly
smooth tame, with a fixed finite loss for each derivative order.
Thus the construction of
\cref{sec:approx} produces the same smooth tame maps $V$ and $\mathscr Q$.
Hamilton's theorem yields
\[
  \mathcal P_{\mathscr F}(u(\mathscr F))=0.
\]
The fixed-range identity then gives
$\mathcal H_{\mathscr F}(u(\mathscr F))=0$. A final $C^1$ shrinking makes
$\Phi_{u(\mathscr F)}$ an equivariant diffeomorphism.
\end{proof}

\appendix

\section{Sobolev interpolation and tame composition}
\label{sec:tame-preliminaries}
The embedding and product conventions are given in
\eqref{eq:sobolev-embedding}--\eqref{eq:tame-product}. Here we record
Sobolev-scale interpolation, prove \cref{prop:composition-evaluation},
and recall finite-dimensional substitution. All estimates apply
componentwise to vector-valued functions.

\begin{Lemma}[Sobolev interpolation]
\label{lem:sobolev-interpolation}
Let $a\le b$ and $0\le\theta\le1$ be real numbers and
$t=(1-\theta)a+\theta b$. For every function $f\in H^b$ we have:
\begin{equation}
  \norm[H^t]{f}
  \le \norm[H^a]{f}^{1-\theta}\norm[H^b]{f}^{\theta}.
\label{eq:sobolev-log-convexity}
\end{equation}
If $a<t<b$, then for every $\delta>0$, we have
\begin{equation}
  \norm[H^t]{f}
  \le \delta\norm[H^b]{f}
  +C_{a,t,b}\delta^{-(t-a)/(b-t)}\norm[H^a]{f}.
\label{eq:sobolev-interpolation-parameter}
\end{equation}
\end{Lemma}
\begin{proof}
For $0<\theta<1$, since
$\norm[H^t]{f}^2=\sum_{\xi\in\Z^N}\langle\xi\rangle^{2t}|\widehat f(\xi)|^2$,
H\"older's inequality with exponents $1/(1-\theta)$ and $1/\theta$ gives
\[
\begin{aligned}
\|f\|_{H^t}^{\,2}
&=
\sum_{\xi\in\Z^N}
\left(
\langle\xi\rangle^{2a}|\widehat f(\xi)|^2
\right)^{1-\theta}
\left(
\langle\xi\rangle^{2b}|\widehat f(\xi)|^2
\right)^\theta
\\
&\le
\left(
\sum_{\xi\in\Z^N}
\langle\xi\rangle^{2a}|\widehat f(\xi)|^2
\right)^{1-\theta}
\left(
\sum_{\xi\in\Z^N}
\langle\xi\rangle^{2b}|\widehat f(\xi)|^2
\right)^\theta
\\
&=
\|f\|_{H^a}^{\,2(1-\theta)}
\|f\|_{H^b}^{\,2\theta}.
\end{aligned}
\]
Taking the square root gives \eqref{eq:sobolev-log-convexity}. The cases
$\theta=0$ and $\theta=1$ are immediate. Combining with Young's
inequality gives \eqref{eq:sobolev-interpolation-parameter}.
See also \cite[\S2, Eqs.~(2.16)--(2.17)]{Moser1988} and
\cite[Chapter~4]{TaylorPDEI2011}.
\end{proof}

\begin{proof}[Proof of \cref{prop:composition-evaluation}]
\emph{(i)} Write $\phi=\Id+u$. Choose $\mathcal V$ so that
$\norm[C^1]{u}$ is bounded by a fixed sufficiently small constant.
Change of variables then gives
$\norm[L^2]{F\circ\phi}\le C\norm[L^2]{F}$ for positive constant $C$, and similarly bounds
composition by $\phi$ on every $L^p$, $2\le p\le\infty$.
For an integer $s\ge s_0$ and $|\alpha|=s$, the chain rule expresses
$\partial^\alpha(F\circ\phi)$ as a finite linear combination, with fixed numerical coefficients, of terms of the form:
\[
 (D^kF\circ \phi)[D^{\beta_1}\phi,\ldots, D^{\beta_k}\phi],
 \qquad
 1\le k\le s,\quad |\beta_j|\ge1,\quad
 \sum_{j=1}^k\beta_j=\alpha.
\]
Set
\[
 \theta_0=\frac{k-1}{s-1},\qquad
 \theta_j=\frac{|\beta_j|-1}{s-1}\quad(1\le j\le k).
\]
Then $\theta_0+\sum_{j=1}^k\theta_j=1$. Whenever $\theta_j=0$, the
corresponding exponent $2/\theta_j$ is interpreted as $\infty$ and the
estimate below is the elementary $L^\infty$ bound. The
Gagliardo--Nirenberg inequalities
\cite{Gagliardo1959,Nirenberg1959}, applied to $DF$ and to the periodic
matrix field $D\phi=I+du$, give
\[
 \norm[L^{2/\theta_0}]{D^kF}
 \le C_s\norm[L^\infty]{DF}^{1-\theta_0}
          \norm[H^s]{F}^{\theta_0},
\]
and
\[
 \norm[L^{2/\theta_j}]{D^{\beta_j}\phi}
 \le C_s\norm[L^\infty]{D\phi}^{1-\theta_j}
          \norm[H^{s-1}]{D\phi}^{\theta_j}.
\]
Combining the H\"older's inequality, the change-of-variables
bound, and
$\norm[H^{s-1}]{D\phi}\le C(1+\norm[H^s]{u})$, we have that:
\[\norm[L^2]{(D^kF\circ \phi)[D^{\beta_1}\phi,\ldots, D^{\beta_k}\phi]}\leq
 C_s\norm[H^s]{F}^{\theta_0}
 \Bigl[\norm[L^\infty]{DF}(1+\norm[H^s]{u})\Bigr]^{1-\theta_0}.
\]
Young's inequality bounds the right-hand-side by
\[
 C_s\Bigl[\norm[H^s]{F}
      +\norm[L^\infty]{DF}(1+\norm[H^s]{u})\Bigr].
\]
Since $s\ge s_0>N/2+1$, Sobolev embedding and the equivalence of
$H^s$ with the $L^2$-norm plus the order-$s$ derivatives give
\eqref{eq:composition-tame}.
Joint smooth tameness follows from
\cite[Part~II, Theorem~2.3.3]{Hamilton1982}. The $C^n$ and $H^n$
gradings on a compact manifold are equivalent up to a fixed shift.
In particular,
\[
 D\bigl[F\circ(\Id+u)\bigr][\dot F,h]
 =\dot F\circ(\Id+u)+(DF)\circ(\Id+u)\,h,
\]
and the higher derivatives are built from the same composition,
differentiation, and product operations.

\emph{(ii)} Choose a fixed cutoff function $\chi$ that is
one on $\overline B_\rho$ and supported in $B_{2\rho}$. Extending by zero to obtain
$\chi\mathscr F$ gives a tame linear map
$\mathscr F\mapsto\widetilde{\mathscr F}$ into functions on a compact
manifold of dimension $N+mq$. Define
\[
  \Psi_u(z)=\bigl(z+u(z),D_Pu(z)\bigr).
\]
Then $u\mapsto\Psi_u$ is smooth tame and of differential order one,
and, on the stated domain,
\[
  \mathcal E(\mathscr F,u)=\widetilde{\mathscr F}\circ\Psi_u.
\]
Hamilton's joint composition theorem applies without requiring
$\Psi_u$ to be a diffeomorphism, proving smooth tameness.
For completeness, choose an integer $\ell>(N+mq)/2$. On sets where
$\norm[C^1]{\mathscr F}+\norm[C^2]{u}\le R$, the $C^s$ composition
estimate \cite[Part~II, Lemma~2.3.4]{Hamilton1982} and Sobolev embedding
give, for every integer $s\ge0$,
\begin{align*}
 \norm[H^s(\T^N)]{\mathcal E(\mathscr F,u)}
 \le C_{s,R,\rho}\Bigl(&1+
 \norm[H^{s+\ell}(\T^N\times\overline B_{2\rho})]{\mathscr F}\\
 &+\norm[H^{s+\ell+1}(\T^N)]{u}\Bigr).
\end{align*}
The product-domain Sobolev norms use any fixed smooth background
metric. The estimate has a fixed, nonoptimal finite loss. Evaluation
of any fixed finite derivative of $\mathscr F$ follows by composing
with the corresponding tame linear differentiation operator.
\end{proof}

\begin{Lemma}
\label{lem:finite-substitution}
Let $\Psi$ be a smooth map on an open subset of a finite-dimensional vector
space.  If $f$ ranges in a set whose image lies in a compact subset of that
open set and whose $H^{s_0}$ norm is bounded for some $s_0>N/2$, then
$f\mapsto\Psi(f)$ is smooth tame.  In particular, matrix inversion,
determinant, and square root on positive numbers are smooth tame on sets uniformly separated from their singular
loci.
\end{Lemma}
\begin{proof}
This is the nonlinear bundle-operator theorem
\cite[Part~II, Theorem~2.2.6, Corollary~2.2.7]{Hamilton1982}.
The Sobolev bounds follow from the nonlinear substitution estimates in
\cite[\S2, Eqs.~(2.22)--(2.23)]{Moser1988}, applied componentwise.
The compact-range hypothesis bounds all required derivatives of
$\Psi$ and, for inversion and square root, keeps the arguments away
from the singular set. Each Fr\'echet derivative is again a bundle
operator applied to $f$ and the variation fields, so it is tame.
\end{proof}

\section{Proof of the commutator estimate}
\label{sec:commutator-proof}
We prove the commutator estimate used in the inverse theorem.  The only
additional ingredient is the following endpoint Gagliardo-Nirenberg
interpolation inequality on the compact torus.

\begin{Lemma}[Endpoint Gagliardo-Nirenberg interpolation]
\label{lem:GN-endpoint}
Let $r\ge2$ be an integer and $1\le k\le r$.  Put
\[
 \theta_k=\frac{k-1}{r-1},\qquad
 \frac1{p_k}=\frac{\theta_k}{2},\qquad
 \frac1{q_k}=\frac{1-\theta_k}{2}.
\]
The endpoint conventions are $p_1=\infty$, $q_1=2$, $p_r=2$, and
$q_r=\infty$.  Then
\begin{align}
 \norm[L^{p_k}]{\partial^k E}
 &\le C_r
 \norm[L^\infty]{\nabla E}^{1-\theta_k}
 \norm[H^r]{E}^{\theta_k},
\label{eq:GN-E}\\
 \norm[L^{q_k}]{\partial^{r-k}v}
 &\le C_r
 \norm[L^\infty]{v}^{\theta_k}
 \norm[H^{r-1}]{v}^{1-\theta_k}.
\label{eq:GN-v}
\end{align}
\end{Lemma}

\begin{proof}
The inequalities are standard instances of the Gagliardo-Nirenberg
interpolation inequality, see \cite{Gagliardo1959,Nirenberg1959}. We record
the exponent calculation because it is what is needed below.  Interpolating between $W^{1,\infty}$ and $H^r$ gives
\[
 k-\frac N{p_k}
 =(1-\theta_k)\cdot1
 +\theta_k\left(r-\frac N2\right).
\]
Since $1+\theta_k(r-1)=k$, this identity is equivalent to
$1/p_k=\theta_k/2$.  Likewise, interpolating between $L^\infty$ and
$H^{r-1}$ gives
\[
 r-k-\frac N{q_k}
 =(1-\theta_k)\left(r-1-\frac N2\right),
\]
which is equivalent to $1/q_k=(1-\theta_k)/2$.  The endpoint cases are the corresponding trivial estimates. Localization from $\R^N$ to the torus does not change the exponents and only changes the constants.
\end{proof}

\begin{proof}[Proof of \cref{lem:commutator}]
When $r=1$ the assertion follows
directly from
\[
 [\partial_j,E]v=(\partial_jE)v
\]
and H\"older's inequality.  We may therefore assume $r\ge2$.  For a
multiindex $\alpha$,
\begin{equation}
 [\partial^\alpha,E]v
 =\sum_{0<\beta\le\alpha}
 \binom{\alpha}{\beta}
 (\partial^\beta E)(\partial^{\alpha-\beta}v).
\label{eq:comm-expansion}
\end{equation}
It is enough to treat $|\alpha|=r$, lower orders are included by increasing
the constant.  Put $k=|\beta|$.  H\"older's inequality with the exponents of
\cref{lem:GN-endpoint} gives
\[
 \norm[L^2]{(\partial^\beta E)
 (\partial^{\alpha-\beta}v)}
 \le
 \norm[L^{p_k}]{\partial^kE}
 \norm[L^{q_k}]{\partial^{r-k}v}.
\]
Using \eqref{eq:GN-E}-\eqref{eq:GN-v}, the right side is at most
\[
 C_r
 \bigl(\norm[L^\infty]{\nabla E}
       \norm[H^{r-1}]{v}\bigr)^{1-\theta_k}
 \bigl(\norm[H^r]{E}
       \norm[L^\infty]{v}\bigr)^{\theta_k}.
\]
Weighted Young's inequality
$X^{1-\theta}Y^\theta\le(1-\theta)X+\theta Y$ therefore yields
\begin{equation}
 \norm[L^2]{(\partial^\beta E)
 (\partial^{\alpha-\beta}v)}
 \le C_r\Bigl(
 \norm[L^\infty]{\nabla E}\norm[H^{r-1}]{v}
 +\norm[H^r]{E}\norm[L^\infty]{v}
 \Bigr).
\label{eq:comm-one-term}
\end{equation}
The formulas remain valid for $k=1$ and $k=r$ by the endpoint conventions.
There are finitely many pairs $(\alpha,\beta)$ for fixed $r$.  Summing
\eqref{eq:comm-one-term} and taking the square sum over $|\alpha|\le r$
proves \eqref{eq:commutator}.  The calculation is componentwise and therefore
applies without change to matrix multiplication.
\end{proof}

\printbibliography

\end{document}